\documentclass[one column,a4paper,10pt]{article}
\usepackage{amsmath,amssymb,color,enumerate,graphicx,mathrsfs,xspace,algorithm,subcaption,algorithmicx,algpseudocode,booktabs,adjustbox,cancel,threeparttable,pdflscape,afterpage,longtable,xargs,listings,amsthm,fullpage,etoolbox}
\usepackage[noblocks,auth-sc]{authblk}
\usepackage[sort&compress,authoryear,round]{natbib}
\usepackage{charter}
\usepackage[math]{blindtext}
\blindmathtrue
\usepackage[dvipsnames]{xcolor}

\usepackage[colorlinks=true, allcolors=NavyBlue]{hyperref}
\usepackage[nameinlink]{cleveref}

\renewcommand{\hat}{\widehat}
\renewcommand{\tilde}{\widetilde}
\renewcommand{\bar}{\overline}
\newtheorem{example}{Example}
\newtheorem{theorem}{Theorem}
\newtheorem{lemma}{Lemma}
\newtheorem{proposition}[lemma]{Proposition}

\theoremstyle{definition}
\newtheorem{definition}{Definition}

\theoremstyle{remark}
\newtheorem{remark}{Remark}
\allowdisplaybreaks

\crefname{assumption}{Assumption}{Assumptions} 
\crefname{lemma}{Lemma}{Lemmata}
\crefname{theorem}{Theorem}{Theorems}
\crefname{proposition}{Proposition}{Propositions}
\crefname{corollary}{Corollary}{Corollaries}
\crefname{claim}{Claim}{Claims}
\crefname{algorithm}{Algorithm}{Algorithms}
\crefname{definition}{Definition}{Definition}
\crefname{remark}{Remark}{Remarks}
\crefname{figure}{Figure}{Figures}
\crefname{section}{Section}{Sections}
\crefname{example}{Example}{Examples}
\crefname{equation}{}{}
\crefname{table}{Table}{Tables}

\newcommand{\R}{\mathbb{R}}

\newcommand{\Z}{\mathbb{Z}}
\newcommand{\B}{\{0,1\}}

\newcommand{\rnd}{\operatorname{rnd}}

\newcommand{\norm}[1]{\left\Vert#1\right\Vert}

\newcommand{\bfp}{{\mathbf{p}}}
\newcommand{\bfu}{{\mathbf{u}}}
\newcommand{\bfv}{{\mathbf{v}}}
\newcommand{\bfw}{{\mathbf{w}}}
\newcommand{\bfx}{{\mathbf{x}}}

\newcommand{\bfy}{{\mathbf{y}}}

\newcommand{\bfz}{{\mathbf{z}}}

\colorlet{myClr}{black}

\newcommand{\Rq}[1][]{ \textcolor{myClr} {\mathbf{R}_{#1}} }

\newcommand{\Qq}[1][]{ \textcolor{myClr} {\mathbf{Q}_{#1}} }

\newcommand{\Qy}{ \Qq[\bfy] }

\newcommand{\bq}[1][]{ \textcolor{myClr} {\mathbf{b}_{#1}} }

\newcommand{\Cq}[1]{ \textcolor{myClr} {\mathbf{C}_{#1}} }
\newcommand{\Cx}{ \Cq{\bfx} }
\newcommand{\Cy}{ \Cq{\bfy} }

\newcommand{\dq}[1][]{ \textcolor{myClr} {\mathbf{d}_{#1}} }
\newcommand{\dx}{ \dq[\bfx] }
\newcommand{\dy}{ \dq[\bfy] }

\newcommand{\fq}[1][]{ \textcolor{myClr} {f_{#1}} }
\newcommand{\fx}{ \fq[\bfx] }
\newcommand{\fy}{ \fq[\bfy] }

\newcommand{\hq}[1][]{ \textcolor{myClr} {h_{#1}} }
\newcommand{\hx}{ \hq[\bfx] }

\newcommandx{\gq}[2][1={},2={}]{ \textcolor{myClr} {g^{#2}_{#1}} }
\newcommand{\gx}{ \gq[\bfx] }
\newcommand{\gy}{ \gq[\bfy] }

\newcommand{\pq}[1][]{ \textcolor{myClr} {\bfp_{#1}} }
\newcommand{\px}{ \dq[\bfx] }
\newcommand{\py}{ \pq[\bfy] }

\newcommand{\sx}{ 0 }
\newcommand{\sy}{ 0 }

\newcommand{\Dq}[1][]{ \textcolor{myClr} {D_{#1}} }
\newcommand{\Dy}{ \Dq[\bfy] }

\newcommandx{\proxn}{\textcolor{black}{\pi}}
\newcommandx{\prox}[2][1={\Qq},2={}]{\textcolor{black}{\bar{\proxn_{#2}} \left( #1 \right) }}
\newcommandx{\vecN}[1][1={}]{\begin{pmatrix}#1\end{pmatrix}}
\newcommandx{\xyVec}[2][1={},2={}]{\vecN{#1x#2\\#1y#2}}

\newcommandx{\ellip}[2][1={\gamma},2={u}] {\textcolor{Orange}{\mathcal{E}_{#1, #2}}}
\newcommand{\Flt}[1]{\textcolor{black}{ \phi _{ #1 }}}
\newcommand{\om}{\star}       
\newcommand{\am}{\dagger}     
\newcommandx{\poDiff}[1][1={\infty}]{\textcolor{black}{\omega_{#1}}}
\newcommandx{\poDiffBest}[1][1={\infty}]{\textcolor{black}{\bar{\omega_{#1}}}}

\newcommand{\SSI}{\textsc{Subset-Sum-Interval}\xspace}

\newcommand{\sigpt}{\Sigma_2^p}

\title{\Huge Proximity-based approximation algorithms for integer bilevel programs}
\author[1]{Sriram Sankaranarayanan \thanks{srirams@iima.ac.in}}
\author[1]{V. Shubha Vatsalya \thanks{phd23shubhav@iima.ac.in}}
\affil[1]{Operations and Decision Sciences, Indian Institute of Management Ahmedabad}
\date{}

\begin{document}
\maketitle
\begin{abstract}
We primarily consider bilevel programs where the lower level is a convex quadratic minimization problem under integer constraints
  We show that it is $\sigpt$-hard to decide if the optimal objective for the leader is lesser than a given value. 
  Following that, we consider a natural algorithm for bilevel programs that is used as a heuristic in practice. 
  Using a result on proximity in convex quadratic minimization, we show that this algorithm provides an additive approximation to the optimal objective value of the leader.
  The additive constant of approximation depends on the flatness constant corresponding to the dimensionality of the follower's decision space and the condition number of the matrix $\Qq$ defining the quadratic term in the follower's objective function. 
    We show computational evidence indicating the speed advantage as well as that the solution quality guarantee is much better than the worst-case bounds.
    We extend these results to the case where the follower solves an integer linear program, with an objective perfectly misaligned with that of the leader.
    Using integer programming proximity, we show that a similar algorithm, used as a heuristic in the literature, provides additive approximation guarantees. 
\end{abstract}

\medskip
\medskip

\section{Introduction}
Bilevel programs model sequential decision-making in strategic games involving two players: the \emph{leader}, who makes decisions first knowing how the second player will respond, and the \emph{follower}, who responds to the leader's decisions.
Both players aim to maximize their own objectives subject to their own constraints. 
The follower has the benefit of \emph{observing} and knowing the decision made by the leader and responding to that; thus, the follower solves a traditional optimization problem. 
In contrast, the leader has to \emph{anticipate} the follower's response and make a decision to maximize their own objective function, given the follower's response.
Even the simplest version of the problem, in which both the leader and the follower solve a linear program, is known to be $NP$-complete \citep{jeroslow_1985}.

Despite the difficulty in solving bilevel programs, they have gathered interest due to a wide range of applications \citep{franceschi18a,caprara_study_2014,Caprara2016,Shen2012,Zhang2018,Bracken1974,Brown2006,Cote2003}.
Given the wide interest in solving bilevel programs, there has also been a recent explosion in the number of algorithms proposed to solve bilevel programs. 
While many of the algorithms are designed to solve certain special subcategories of bilevel programs, there have also been a few general-purpose solvers. 
Recently, \citet{Fischetti2017} provided an algorithm to solve linear bilevel programs with linear and integer constraints for both the leader and the follower.

When both the leader and the follower have linear and integer constraints, with linear objectives, the problem belongs to a complexity class called $\sigpt$ and is \emph{complete} for this class of decision problems. 
In other words, it is conjectured that given oracle access to solve $NP$-hard problems, one would need to use the oracle exponentially many times, or input a problem to the oracle that is exponentially large in size, to solve the bilevel program.
This gives a strong indication that when integer variables are involved, bilevel programs are really difficult to solve, even with the aid of off-the-shelf integer-programming ($NP$-complete problem) solvers like Gurobi \citep{GurobiOptimization2023} or CPLEX \citep{cplex2009v12}.

A typical approach in the literature, when problems can be provably hard to solve, is to develop \emph{approximation algorithms}.
These are algorithms which has a much better run-time guarantee, but solve the problem approximately, \emph{i.e., } the solution obtained is not guaranteed to be optimal.
However, the error induced by the approximation is bounded, meaning that there is a quality guarantee on the solution obtained.
More formally, typical approximation algorithms developed for $NP$-complete problems run in polynomial time and provide a solution of provable quality.
Since, the problem of interest here are known to be $\sigpt$-hard, we develop approximation algorithms which use an oracle to solve $NP$-complete problems, but only calls them at most polynomially many times.
This approach is also practical because standard solvers like Gurobi \citep{GurobiOptimization2023}, CPLEX \citep{cplex2009v12} etc., can efficiently solve a wide range of $NP$-complete problems of reasonable sizes, making them suitable \emph{oracles}. 

In this spirit, we provide the following contributions. 
First, we extend the $\sigpt$-hardness results known in the literature. 
We already know that if the follower optimizes a linear function with integer as well as \emph{linear} constraints, the problem is $\sigpt$-hard.  
    One can readily notice that if the follower has no linear constraints but retains a linear objective, the follower's problem becomes trivially unbounded. 
    Thus, the next interesting case involves having no linear constraints, but only a convex-quadratic objective.
    Therefore,  we consider a family of bilevel programs where the follower minimizes a convex-quadratic function subject only to integer constraints. 
    The leader has a linear objective function and linear constraints. 
    We prove that deciding if the leader can achieve an objective $\gamma$ in this case is $\sigpt$-hard. 

    Next, we consider a very natural algorithm for bilevel programs where the follower has integer constraints: (1) Solve the bilevel program after relaxing the follower's integer constraints; (2) Freeze the leader's optimal solution for this relaxed version; (3) Solve the follower's problem optimally; (4) Return the obtained leader-follower solution pair. 
    This algorithm is considered a fairly standard heuristic in the literature. 
    For example, \citet{Caprara2016} uses this for their algorithm to solve bilevel knapsack problems while \citet{Fischetti2018} extends it for generalized interdiction problems. 
    Nevertheless, for completeness, we state this algorithm formally in \cref{alg:apx}. 
        The algorithm relies on the relative ease of solving the follower's problem {\em without} the integer constraints. 
    This is true especially when convexity or other such properties hold for the follower's relaxed optimization problem, and thus KKT condition-based and other reformulation techniques make the problem solvable. 
    We show that this algorithm provides an additive approximation guarantee for the leader's objective function.
    Our results provide a theoretical explanation for \emph{why} these heuristics work well in practice, and also provide bounds and guarantees on the quality of the solution obtained.

The approximation guarantees we provide are based on the concept of \emph{proximity}.
Proximity, typically in the context of integer programs, refers to the maximum distance between an integer optimal solution and the optimal solution of the continuous relaxation. 
Although definitions vary slightly, this concept is well-studied in the theoretical integer programming literature. 
All our approximation results depend on proximity guarantees for various types of problems, and our bounds will improve as there are advances in the theory of proximity.

\section{Literature Review}
Bilevel programs have a wide range of applications. 
For example, \citet{franceschi18a} propose using bilevel programs for hyperparameter optimization in machine learning. 
Many interdiction problems naturally take the form of a bilevel program \citep{caprara_study_2014,Caprara2016,Shen2012,Zhang2018}. 
Naturally, this leads to applications in defense. 
\citet{Bracken1974} suggest various defense and security applications of bilevel programs, which were further extended in \citet{Brown2006} for defending critical infrastructure. 
\citet{Cote2003} use bilevel programs to solve a pricing problem in the airline industry. 

Theoretical analysis of bilevel programs have received significant attention in the literature starting from the seminal work of \citet{Jeroslow1985}, who showed that bilevel programs are $NP$-complete.
Despite the $NP$-hardness established by Jeroslow, algorithms to solve the bilevel program have been available for a long time.
\citet{edmunds1992algorithm} present an algorithm to find the exact solution for a bilevel program where the leader has mixed integer variables and a convex nonlinear objective function, while the follower has continuous variables with a convex quadratic objective function. 

Mixed-integer bilevel linear programs are known to be $\sigpt$-hard. 
\citet{caprara_study_2014} show that several restricted versions of these programs are already $\sigpt$-hard. 
Therefore, polynomial-time algorithms for solving mixed-integer bilevel programs, even with oracle access to solve $NP$-complete problems, are unlikely to exist.

Given the strong negative results,the literature review is divided into three parts. First, we discuss enumerative algorithms with minimal guarantees on the worst-case run time, but providing exact solutions. Second, we explore approximation algorithms offering provable solution quality, and with guarantees on the worst-case run time. 
Finally, we discuss fixed-parameter tractable algorithms for bilevel programs. 

\citet{Fischetti2017} propose an exact solution using branch and cut for mixed-integer bilevel programs with linear objectives and constraints for both the leader and the follower. 
\citet{Fischetti2018a} develop intersection cuts, a family of valid inequality originally created by \citet{Balas1971} for integer programs, and show that they can be used to solve mixed-integer bilevel programs much faster. 
On a similar note, \citet{audet2007new,tang2016class,tahernejad2020branch,denegre2009branch,xu2014exact} also develop algorithmic frameworks that solve a wide range of mixed-integer bilevel linear problems. 
More recently, \citet{Gaar2023} provide valid inequalities for mixed-integer \emph{nonlinear} bilevel programs based on a disjunctive cut formulation. 
These algorithms primarily operate within a branch-and-cut framework for mixed-integer bilevel programs. 
Complementarily, \citet{lozano2017value} developed an exact finite algorithm to solve bilevel mixed-integer programs with nonlinear constraints and objective dependent on both leader and follower variables. The leader's variables are integral, while the follower can have both integer and continuous variables.  
Their approach is based on iteratively approximating the value function of the follower, using simultaneous row and column generation. 
\citet{anandalingam1990solution} use penalty functions to develop an exact algorithm to solve linearly constrained Stackelberg problems. 
\citet{caramia2015enhanced} propose two new enhanced exact algorithms to solve discrete (linear) bilevel programs, \emph{i.e., }bilevel programs where the leader and the follower have only integral variables.
\citet{Wang2017} is another exact algorithm based on generalized branching to solve discrete linear bilevel programs.

Bilevel knapsack problems, a subset of bilevel linear programs, are of particular interest. 
\citet{Caprara2016} developed an algorithm that yields the exact solution for a bilevel integer program where the leader's decision interdicts a subset of knapsack items for the follower. 
\citet{brotcorne2009dynamic} presented a dynamic programming algorithm that solves a bilevel knapsack problem where the leader determines the capacity of the knapsack and the follower solves the knapsack problem. 
\citet{ghatkar2023solution} adapted two enumerative algorithms to solve a specific type of mixed-integer bilevel knapsack problems where the leader has continuous variables and a non-linear interaction with the follower's variables, which are both continuous and discrete.

When structured nonlinearities are allowed, the literature focsuses more towards problems that typically have convex nonlinearities in the follower's objective, but not any integer constraints on the follower. 
For example, recently, \citet{Byeon2022} provide an exact solution for bilevel programs where the follower solves a second-order conic program, using the Bender's decomposition method. 
\citet{mitsos2010global} created an algorithm to find the global solution of nonlinear bilevel mixed-integer programs using a convergent lower bound.
\citet{Kleinert2021} developed a cutting plane algorithm based on multi and single tree outer approximation to find approximate solutions to bilevel programs which have integer variables for the leader but the follower solves a convex quadratic program.  
\citet{Qi2024} consider competitive facility location problems, which have nonlinearity in the lower level. They provide various algorithms to solve the problem exactly as well as approximately. 

There has been relatively less work on approximation algorithms for bilevel programs. 
For the special case of bilevel knapsack problems, \citet{Chen2013} provides various constant factor approximation algorithms (factor of $2+\varepsilon$, $1+\sqrt{2}+\varepsilon$ for different versions).
For bilevel knapsack problems with interdiction constraints, \citet{caprara_study_2014} provide a \emph{polynomial-time approximation scheme} (PTAS). 
\citet{Qi2024}, in their study of bilevel programs under competitive facility location, provide an approximation algorithm, which makes calls to an integer programming solver, \emph{i.e., } an oracle to solve an $NP$-complete problem. 

Finally, 
\citet{koppe2010parametric} used results from parametric integer programming to develop fixed-parameter tractable algorithms to solve a class of bilevel integer linear programs.

\section{Problem Formulation}
 Bilevel programs are problems with two players - the leader and the follower.
 The leader makes a decision \emph{anticipating} the follower's response, while the follower makes a decision after \emph{observing} the leader's decision.
The literature distinguishes between \emph{two} versions of the bilevel program -- the optimistic and the pessimistic versions -- depending on how the follower chooses a decision if they have multiple optimal responses to the leader. 
In the optimistic version, it is assumed that among multiple optimal solutions, the follower will choose the one which favors the leader the most, while in the pessimistic version, the follower will choose the one which hurts the leader the most. 
It is also common to call the leader's problem the upper-level problem and the follower's problem the lower-level problem.
We define both formally below. In the definitions, we use generic functions $\fx, \fy, \gx$ and $\gy$. We will make assumptions on these functions soon, as needed.
\begin{definition}[Integer Convex Bilevel Programs] 
        An Optimistic Integer Convex Bilevel Program  is a problem of the form 
 \begin{align}
   \min_{\bfx\in \Z^{n_x},\bfy \in \Z^{n_y}} \quad&:\quad \fx (\bfx,\bfy) &\text{s.t.} \nonumber \\
   \gx^i (\bfx) \quad&\leq\quad 0 \qquad \forall\ i \in \{1,\ldots,m_x\} \\
   \bfy \quad&\in\quad \arg \min_{\bfy \in \Z^{n_y}} \left\{ \fy (\bfx,\bfy) \,:\, \gy^i (\bfx,\bfy) \,\leq\, 0 \quad \forall\ i \in \{1,\ldots,m_y\} \right\} \tag{ICB-Gen-O} \label{eq:Sklbrg-Gen-O}
\end{align}
  A Pessimistic Integer Convex Bilevel Program  is a problem of the form 
 \begin{align}
   \min_{\bfx\in \Z^{n_x}} \max_{\bfy \in \Z^{n_y}} \quad&:\quad \fx (\bfx,\bfy) &\text{s.t.} \nonumber \\
   \gx^i (\bfx) \quad&\leq\quad \sx \qquad \forall\ i \in \{1,\ldots,m_x\} \nonumber \\
   \bfy \quad&\in\quad \arg \min_{\bfy \in \Z^{n_y}} \left\{ \fy (\bfx,\bfy) \,:\, \gy^i (\bfx,\bfy) \,\leq\, \sy \quad \forall\ i \in \{1,\ldots,m_y\} \right\} \tag{ICB-Gen-P} \label{eq:Sklbrg-Gen-P}
 \end{align}
 
Here $\fx, \fy, \gx^i$ for $i=1,\ldots,m_x$ and $\gy^i$ for $i=1,\ldots,m_y$ are convex functions of $(\bfx,\bfy)$. We call the player who decides $\bfx$ the leader, and the player who decides $\bfy$ the follower.
\end{definition}

In a bilevel program, using the notation as above, the leader chooses $\bfx \in \Z^{n_x}$ first. Observing the leader's choice,  the follower solves an optimisation problem parametrised in $\bfx$, to decide their variables $\bfy$. In turn, the leader's objective function and feasible set themselves are parametrised in $\bfy$, and they anticipate $\bfy$'s behaviour to choose $\bfx$.
Now, given a leader's decision $\bfx$, it is possible that the follower has multiple optimal solutions. 
If the leader can control or influence the follower to choose a decision that is to their liking, among the multiple optimal solutions that the follower has, it is said to be the optimistic version of the problem.
Equivalently, in the optimistic version, among the multiple optima, the follower chooses $\bfy$ which is the most favourable to the leader. 
On the other hand, if the leader has little to no influence on the follower, then the leader should possibly plan for the worst alternative that the follower might pick, giving rise to the pessimistic version of the problem.
So, in the pessimistic version,  we assume that the follower chooses the solution which is the least favourable to the leader. 
The solutions to each of the versions could be considerably different. 
The example below highlights this difference.
\begin{example} 
Consider the following problem
    \begin{subequations}
        \begin{align*}
            \min_{\bfx\in \Z} \max_{\bfy \in \Z} \quad&:\quad 100(\bfx-1)^2 + \bfy &\text{s.t.} \nonumber \\
   \bfy \quad&\in\quad \arg \min_{\bfy \in \Z} \left\{ \bfy^2 - \bfx\bfy \right\}
        \end{align*} 
    \end{subequations}
and
    \begin{subequations}
        \begin{align*}
            \min_{\bfx\in \Z} \min_{\bfy \in \Z} \quad&:\quad 100(\bfx-1)^2 + \bfy &\text{s.t.} \nonumber \\
   \bfy \quad&\in\quad \arg \min_{\bfy \in \Z} \left\{ \bfy^2 - \bfx\bfy \right\} 
        \end{align*} 
    \end{subequations}
The former is the pessimistic version of the problem, while the latter is the optimistic version of the problem. 
Observe that if the leader chooses an $\bfx$ that is different from $1$, then their objective value will grow fast. 
Thus, the leader always chooses $\bfx=1$.
Given $\bfx=1$, the follower has two optimal solutions: $\bfy=0$ and $\bfy=1$.
While the follower is indifferent between these choices, the leader prefers $\bfy=0$. Thus, the solution to the optimistic version is $(1, 0)$ with the leader's objective value of $0$, and the solution to the pessimistic version is $(1, 1)$ with the leader's objective value being $1$.
\end{example}

In this paper, while we are interested in a bilevel program where the follower has integer constraints in the problem, a useful sub-problem is the one where the integer constraints of the follower are relaxed. It is well known that if the follower's problem is relaxed, then the resulting problem is neither a relaxation nor a restriction of the original bilevel program. 
The following example demonstrates this. 
\begin{example}
    Consider the following bilevel program.
    \begin{subequations}
        \begin{align*}
            \min_{\bfx, \bfy \in \Z} \quad&:\quad 1000\bfx^2 + \bfy & \text{s.t.}\\
            \bfy \quad&\in\quad \arg \min_{\bfy \in \Z}
            \{ (3\bfy - 2)^2 + \bfx \bfy 
            \}
        \end{align*}
    Clearly, the leader chooses $\bfx = 0$. Suppose we relax the follower's problem. Given the leader's decision, the follower's optimal solution is $\bfy = \frac{2}{3}$. The corresponding objective value for the leader is $\frac{2}{3}$. 
    One can also easily observe that the optimal solution of the follower {\em with} the integer constraints, given the leader's decision,  is $\bfy = 1$. The corresponding objective value for the leader is ${1}$, which is worse for the leader, indicating that relaxing the integer constraints of the follower led to a worse outcome for the leader. 
        
        In contrast, consider the following problem.
        
        \begin{align*}
            \min_{\bfx, \bfy \in \Z} \quad&:\quad 1000\bfx^2 + \bfy & \text{s.t.}\\
            \bfy \quad&\in\quad \arg \min_{\bfy \in \Z}
            \{ (3\bfy - 1)^2 + \bfx \bfy 
            \}
        \end{align*}
    Clearly, the leader chooses $\bfx = 0$. Suppose we relax the follower's problem. Given the leader's decision, the follower's optimal solution is $\bfy = \frac{1}{3}$. The corresponding objective value for the leader is $\frac{1}{3}$. 
    One can readily observe that the optimal solution of the follower {\em with} the integer constraints, given the leader's decision,  is $\bfy = 0$. The corresponding objective value for the leader is ${0}$, which is better for the leader! 
    \end{subequations}
\end{example}

While it is clear that relaxing the integer constraints of the follower and solving the resultant (easier) bilevel program provides neither an upper bound nor a lower bound for the original bilevel program, it could still be useful in practice. 
We define this relaxed version of the bilevel program formally.
\begin{definition}[Follower relaxed version (FRV)]
    Given a pessimistic or optimistic version of the bilevel program, the follower relaxed version of the problem is the one where the integer constraints of the follower's problem are relaxed. In the FRV, we retain the pessimistic or optimistic assumptions as in the original bilevel program.
\end{definition}
For example, the FRV of \cref{eq:Sklbrg-Gen-O} is 
 \begin{align*}
   \min_{\bfx\in \Z^{n_x},\bfy \in \R^{n_y}} \quad&:\quad \fx (\bfx,\bfy) &\text{s.t.} \nonumber \\
   \gx^i (\bfx) \quad&\leq\quad 0 \qquad \forall\ i \in \{1,\ldots,m_x\} \\
   \bfy \quad&\in\quad \arg \min_{\bfy \in \R^{n_y}} \left\{ \fy (\bfx,\bfy) \,:\, \gy^i(\bfx,\bfy) \,\leq\, \sy \quad  \forall\ i \in \{1,\ldots,m_y\} \right\} 
\end{align*}

Typically, solving the FRVs of a bilevel program is likely to be easier than solving the original \cref{eq:Sklbrg-Gen-O} or \cref{eq:Sklbrg-Gen-P}. 
For example, we note that if $\fx$ and $ \fy$ are linear and the regions defined by $\gx(\bfx, \bfy) \le 0$ and $\gy(\bfx,\bfy) \le 0$ are polyhedra, then the FRVs are NP-complete, whereas the original \cref{eq:Sklbrg-Gen-P,eq:Sklbrg-Gen-O} are $\sigpt$-complete. 
Moreover, in the linear versions, there are integer-programming reformulations based on big-M, or reformulations into complementarity problems that can be solved readily by off-the-shelf integer programming solvers. 
However, analogous solvers for problems {\em with} integer constraints on the follower are yet to mature.
In some cases, the FRVs are likely to have polynomial-time approximation algorithms, whereas the original version of the problem is not known to have any such approximation algorithm.

In this paper, we discuss two variations of Integer Convex Bilevel Programs,  categorised based on how the leader and follower interact.
In the first variation, the leader's decision variable appears only in the follower's convex quadratic objective function, and the follower has no constraints except for the integer constraint. 
In the second variation, the leader's decision variable appears only in the follower's constraints. The follower has a linear objective function which is independent of the leader's decision. 
We construct feasible solutions to these specialised variations by temporarily relaxing the integer constraint on the follower (i.e., solving the FRV), and then prove that such solutions approximate the actual solutions in various settings.

\section{Follower with a convex quadratic objective}\label{sec:Quad}
In this section, we consider settings where the follower solves a convex quadratic problem with no constraints other than the integer constraint, while the leader has a linear objective function. Interaction between the leader and follower occurs exclusively through the objective function. 

\begin{definition}[Integer Convex Quadratic Bilevel Program]
  An Integer Convex Quadratic Bilevel Program is a variation of \cref{eq:Sklbrg-Gen-P} or \cref{eq:Sklbrg-Gen-O} 
 where the leader's objective $\fx(\bfx,\bfy) = \hx(\bfx) + \dx (\bfy)$ is a linear function; 
 the leader's feasible set is independent of the follower's variables, 
 i.e., $\gx(\bfx,\bfy) = \gx (\bfx)$; the follower's objective $\fy(\bfx,\bfy) = \frac{1}{2}\bfy^\top \Qy \bfy + \left( \Cy \bfx + \dy \right)^\top \bfy  $ is a (strictly) convex quadratic function with the leader's variable appearing only in the linear term, 
 and the follower has no constraints, i.e., $\gy(\bfx, \bfy) = -1$.
\end{definition}

The Pessimistic Integer Convex Quadratic Bilevel Program is of the form: 
\begin{align}
   \min_{\bfx\in \Z^{n_x}} \max_{\bfy \in \Z^{n_y}} \quad&:\quad \hx (\bfx) + \dx (\bfy) &\text{s.t.} \nonumber \\
   \gx^i (\bfx) \quad&\leq\quad \sx \quad \forall\ i \in \{1,\ldots,m_x\} \nonumber \\
   \bfy \quad&\in\quad \arg \min_{\bfy \in \Z^{n_y}} \left\{ \frac{1}{2}\bfy^\top \Qy \bfy + \left( \Cy \bfx \right)^\top \bfy + \dy^\top \bfy \right\} \tag{ICB-Quad-P} \label{eq:Sklbrg-Quad-P}
 \end{align}
And the Optimistic Integer Convex Quadratic Bilevel Program is of the form:

\begin{align}
   \min_{\bfx\in \Z^{n_x}, \bfy \in \Z^{n_y}} \quad&:\quad \hx (\bfx) + \dx (\bfy) &\text{s.t.} \nonumber \\
   \gx^i (\bfx) \quad&\leq\quad \sx \quad \forall\ i \in \{1,\ldots,m_x\} \nonumber \\
   \bfy \quad&\in\quad \arg \min_{\bfy \in \Z^{n_y}} \left\{ \frac{1}{2}\bfy^\top \Qy \bfy + \left( \Cy \bfx \right)^\top \bfy + \dy^\top \bfy \right\} \tag{ICB-Quad-O} \label{eq:Sklbrg-Quad-O}
 \end{align}

\subsection{Complexity}
First, we quantify the computational complexity of the pessimistic version of the Integer Convex Quadratic Bilevel Program. 
We prove in this section that the decision problem of determining if the leader's optimal objective value is $\gamma$ is $\sigpt$-hard.
This means that the problem is in the second level in the polynomial hierarchy of complexity classes. 
This further indicates that unless $NP=\sigpt$, which is a very unlikely situation in complexity theory, there are no algorithms that can run in polynomial time even with access to an oracle that solves $NP$-complete problems in polynomial time. 
In other words, unless very unlikely results in complexity theory hold, one might have to solve either an exponentially large $NP$-hard problem, or solve exponentially many $NP$-complete problems, in order to solve a $\sigpt$-hard problem. 
We formally state the result below.
\begin{theorem}\label{thm:QuadIntHard}
    The pessimistic version of the \em{Integer Convex Quadratic Bilevel Program} as stated in \cref{eq:Sklbrg-Quad-P} is $\sigpt$-hard.
\end{theorem}

Complexity proofs are typically provided by reducing a different (hard) problem to the given problem. 
This way, if the given problem is solved ``fast'', then so is the original hard problem. 
Thus, using a contrapositive argument, we can say that if the original hard problem is unlikely to be solved fast, then the given problem is also unlikely to be solved fast.
Here, we show that a problem called \SSI, that is known to be $\sigpt$-hard can be reduced into the pessimistic version of the Integer Convex Quadratic Bilevel Program. 
To this end, we define \SSI and formally state its hardness below.
\begin{definition}[\SSI]
\SSI is the following decision problem:
\begin{quote}
    Given a sequence $q_1,...q_k$ of positive integers and two more positive integers $R$ and $r$, such that $r \le k$, decide if there exists $S$ such that $R \le S < R+2^r$, such that there does not exist a subset $I \subseteq \{1,2,\ldots,k\}$, $\sum_{i\in I}q_i = S$.
\end{quote}
\end{definition}

One can observe that if such an $S$ is given, checking that $\sum_{i\in I}q_i \neq S$ for every subset $I$ is already NP-hard, as it is an instance of the \textsc{Subset-Sum} problem. \citet{eggermont2013} showed that \SSI belongs to the complexity class $\sigpt$, and it is indeed complete for this class. We formally state this below.
\begin{proposition}[\citep{eggermont2013}]
    \SSI is $\sigpt$-complete. 
\end{proposition}

Now, we are in a position to prove \cref{thm:QuadIntHard}.
The proof of this theorem is motivated by the proof in \citet[Sec 3.3]{caprara_study_2014} where a certain version (DNeg) of bilevel knapsack problem is $\sigpt$-hard.
While significant book-keeping is required for the proof, we use notations involving superscripts $p$, $d$ and $o$ to hint at the analogous forms of {\em padding items}, {\em dummy items} and {\em ordinary items} used in \citet[Sec 3.3]{caprara_study_2014}.
We also note that the proof in this paper does not trivially follow from the complexity results in \citet[Sec 3.3]{caprara_study_2014}, as the follower is not allowed to have any constraints besides the integer constraints in our context.

\begin{proof}[Proof of \cref{{thm:QuadIntHard}}]
To prove \cref{thm:QuadIntHard}, we show that an instance of \SSI can be rewritten as an instance of \cref{eq:Sklbrg-Quad-P} of size bounded by a polynomial in the size of the \SSI instance. Thus, if \cref{eq:Sklbrg-Quad-P} can be solved fast, then this reformulation can be used to solve the \SSI fast, which is unlikely. 
Moreover, we exclude \SSI instances with the trivial case where $R$ cannot be represented as a subset sum of $q_1$ to $q_k$.
This can be checked with an oracle to the \textsc{Subset-sum} problem (which is NP-complete), and does not alter the complexity of \SSI.

\paragraph{A technique to include binary follower variables.}
Given an instance of \SSI, we construct an instance of \cref{eq:Sklbrg-Quad-P}, where some of the follower variables are restricted to be binary. 
However, since the problem of interest, \cref{eq:Sklbrg-Quad-P}, does not allow any constraints on the follower except the integrality constraints, we need to appropriately penalize deviation from binary values.
A variable $\bfy_i$ can be made binary by including the term $M^2(\bfy_i^2-\bfy_i)$ in the objective function, for a sufficiently large value of $M$. For $\bfy_i = 0$ or $1$, this expression evaluates to $0$, and for any other integer value of $\bfy_i$, the expression evaluates to a large number which is at least $M^2$. Thus, by choosing large values of $M$, the follower could be forced to choose binary values for $\bfy_i$. Furthermore, adding the said term $M^2(\bfy_i^2-\bfy_i)$ never introduces any non-convexity in the problem, as $M(\bfy_i^2-\bfy_i)$ is a convex quadratic function of $\bfy_i$.
Thus for the rest of the proof, we arbitrarily allow any follower variable to be binary. 

\paragraph{Defining the instance of \cref{eq:Sklbrg-Quad-P}.}
First, given an instance of \SSI with $q_1,\ldots,q_k$, we define the constants $Q := \sum_{i=1}^k q_i$, and $B:=R+2^r-1+rQ$.
Next, we define $4r+2k+1$ variables controlled by the follower. 
The first set of variables notated as $\bfy^p \in \B^r$, $\bfy^d \in \B^r$, $\bfy^o \in \B^k$, and  $\bfy_s \in \B$ account for $2r+k+1$ variables.
The remaining $2r+k$ variables are notated as $\bfz^p \in \B^r, \bfz^d \in \B^r$ and $\bfz^o \in \B^k$.
For notational convenience, the variables $\bfy^p, \bfz^p, \bfy^d$ and $\bfz^d$ are indexed from $0$ to $r-1$. $\bfy^o$ and $\bfz^o$ are indexed from $1$ to $k$.
The leader has $2r+k$ variables  notated as $\bfx^p \in \B^r$  and  $\bfx^d \in \B^r$
The objective function of the follower is 
\begin{subequations}
\begin{gather}
 \fy^{\text{binary}}(\bfy) + 
    \left (
\sum_{i=0}^{r-1} \left(
(Q+2^i)\bfy^p_i + Q\bfy^d_i
\right) + 
\sum_{i=1}^{k} q_i \bfy^o_i + \bfy_s - B
\right)^2 + \nonumber \\
M\sum_{i=0}^{r-1}\left (
       \bfx_i^p+\bfy_i^p-\bfz_i^p
\right)^2 + 
M\sum_{i=0}^{r-1}\left (
       \bfx_i^d+\bfy_i^d-\bfz_i^d
   \right)^2,  \label{eq:QuadHardProb:FollowerObj}
\end{gather}
where $ \fy^{\text{binary}}(\bfy) $ correspond to those terms which enforce the binary constraints on the required follower's variables.
As needed in \cref{eq:Sklbrg-Quad-P}, all the decision variables of the follower are constrained to be integers.

The leader's constraints are the following.
    \begin{align}
        \sum_{i=0}^{r-1}\bfx^p_i + \sum_{i=0}^{r-1}\bfx^d_i \quad&\le\quad r \\
        \bfx^p_i, \bfx^d_i \quad&\in\quad \B,
    \end{align}
and the leader's objective is to minimise
\begin{align}
    \sum_{i=0}^{r-1} (Q+2^i)\bfy^p_i +         \sum_{i=0}^{r-1}  Q\bfy^d_i + \sum_{i=1}^{k} q_i \bfy^o_i
  \label{eq:QuadHardProb:LeaderObj}
\end{align} \label{eq:QuadHardProb}
\end{subequations}

Having defined the equivalent Integer Convex Quadratic Bilevel Program in \cref{eq:QuadHardProb}, we first analyse the follower's objective function.
The last term as mentioned in the objective function ensures that $\bfx^d_i + \bfy^d _i - \bfz^d_i = 0$. If not, the objective increases by $M$, which the follower does not want.
Since $\bfz^d_i$ is non-negative (binary in particular), and it does not appear anywhere, we can simplify this to indicate a follower constraint $\bfx^d_i + \bfy^d_i \le 1$ for each $i$ from $0$ to $r-1$,

Analogously, by looking at the penultimate term, we can think of it to be equivalent to $\bfx^p_i + \bfy^p_i \le 1$ constraint added to the follower's problem.

Finally, the first term is minimised when 
\begin{align*}
     \sum_{i=0}^{r-1} (Q+2^i)\bfy^p_i +         \sum_{i=0}^{r-1}  Q\bfy^d_i + \sum_{i=1}^{k} q_i \bfy^o_i  \quad&=\quad  B-\bfy_s
\end{align*}
The LHS here is exactly the leader's objective. 
$\bfy_s$ in the RHS can be chosen by the follower to be $0$ or $1$. 
Thus, this term is minimised if the leader's objective is either $B$ or $B-1$.
Moreover, the pessimistic assumption ensures that, given both the choices, the follower will try to ensure that the leader's objective is $B$ and not $B-1$.

\paragraph{The equivalent bilevel program.}
Putting the above observations together, the essential bilevel program that gets solved is equivalent to the following.
\begin{subequations}
    \begin{align}
        \min_\bfx \max_\bfy \quad&:\quad \sum_{i=0}^{r-1} (Q+2^i)\bfy^p_i +         \sum_{i=0}^{r-1}  Q\bfy^d_i + \sum_{i=1}^{k} q_i \bfy^o_i & \text{s.t.}\\
        \sum_{i=0}^{r-1}\bfx^p_i + \sum_{i=0}^{r-1}\bfx^d_i \quad&\le\quad r \\
        \bfx^p_i, \bfx^d_i \quad&\in\quad \B \\
        (\bfy^o,\bfy^p,\bfy^d) \quad&\in\quad \arg\max_\bfy \left \{
        \begin{aligned}
      1
      &\qquad \text{s.t.}\\
            \bfx^p_i + \bfy^p_i \quad&\le \quad 1 \\
            \bfx^d_i + \bfy^d_i \quad&\le \quad 1 \\
            \bfx^o_i + \bfy^o_i \quad&\le \quad 1 \\
            \sum_{i=0}^{r-1} (Q+2^i)\bfy^p_i +         \sum_{i=0}^{r-1}  Q\bfy^d_i + \sum_{i=1}^{k} q_i \bfy^o_i  \quad&\le \quad B\\
            \sum_{i=0}^{r-1} (Q+2^i)\bfy^p_i +         \sum_{i=0}^{r-1}  Q\bfy^d_i + \sum_{i=1}^{k} q_i \bfy^o_i  \quad&\ge \quad B-1
        \end{aligned}
        \right \}
\end{align} \label{eq:QuadHardProb:b}
\end{subequations}
The follower's ability to constrain the leader's objective to be either $B$ or $B-1$ is captured by the last two constraints of the follower.
We let the follower's objective to be $1$, indicating that the follower has no objective. 
The pessimistic assumption now ensures that among multiple feasible solutions, the follower will choose the one that hurts the leader the most. 
Next, we state that achieving an optimal objective value strictly better than $B$ indicates we have a \textsc{Yes} instance of \SSI.

\paragraph{Claim 1. \SSI \textsc{Yes} $\implies$. } If the answer to the \SSI instance is \textsc{Yes}, then the leader's optimal objective value for the problem in \cref{eq:QuadHardProb:b} is $B-1$ or smaller.

\noindent {\em Proof of claim 1.} 
Let $S$ be the smallest integer between $R$ and $R+2^r$ that cannot be written as a subset sum of $q_1$ to $q_k$. This means $S$ cannot be written as a subset sum of $q_1$ to $q_k$ but $S-1$ can be. 
The constraints on the leader restricts them to choose at most $r$ components of the $\bfx^p$ and $\bfx^d$ variables to be $1$.
We build a feasible solution to the leader that has objective value of at most $B-1$ now.
Since $S-R$ is a positive integer less than $2^r$, it can be written as a sum of first $r$ powers of $2$.
i.e., $S-R = \sum_{i\in I} 2^i$ where $I \subseteq \{0, \dots, r-1\}$. 
Now, let the leader choose $\bfx_i^p = 1, \bfx_i^d = 0$ if $i \in I$ and $\bfx_i^p=0, \bfx_i^d = 1$ if $i \not\in I$.
Clearly, the leader has picked exactly $r$ components of $\bfx^p, \bfx^d$ to be $1$. 

As analysed earlier, the follower essentially has constraints of the form $\bfx^p_i + \bfy^p_i \le 1$ and $\bfx^d_i + \bfy^d_i \le 1$.
Observe that among the $2r$ components in $\bfx^p$ and $\bfx^d$, the leader has already set $r$ of them to $1$. 
This means, at most the remaining $r$ of the components in $\bfy^p$ and $\bfy^d$ could be set to $1$ by the follower. 
If the follower sets anything fewer than $r$ components to $1$ (let's say $\ell < r$ components), then the largest value that the leader's objective could possibly make is at most $\ell Q + (2^r-1) - (S-R) + Q $. The last $+Q$ is due to setting all of {\em $\bfy^o_i$s} to $1$. However, this objective is strictly less than $B$ as needed. 
So, if this is optimal to follower, then we are already done for the claim.
However, this need not be optimal to the follower. 
Alternatively, if the follower sets all $r$ components to $1$, then the first two terms, i.e., $\sum_{i=0}^{r-1} (Q+2^i)\bfy^p_i +         \sum_{i=0}^{r-1}  Q\bfy^d_i$ evaluates to exactly $rQ + (2^r-1) + (R-S) = B-S$. Note that we said $S$ cannot be written as subset sum, but $S-1$ can be written as a subset sum of $q_1$ to $q_k$.
This means, the following. Suppose $\sum_{i\in I}q_i = S-1$ for some $I \subseteq \{1,\ldots,k\}$, then setting $\bfy^o_i = 1$ for $i\in I$ and the remaining $\bfy_i^o = 0$ makes $\sum_{i=0}^{r-1} (Q+2^i)\bfy^p_i +         \sum_{i=0}^{r-1}  Q\bfy^d_i + \sum_{i=1}^{k} q_i \bfy^o_i = B-1$. Now, the follower can choose $\bfy_s = 1$, which will be optimal for the follower. However, this means that the leader has an objective of $B-1$ as required.

\paragraph{Claim 2. \SSI \textsc{No} $\implies$.} If the answer to the \SSI instance is \textsc{No}, then the leader's optimal objective value for the problem in \cref{eq:QuadHardProb:b} is $B$.

\noindent {\em Proof of claim 2.} The leader must choose at most $r$ components of $\bfx^p$ and $\bfx^d$ to be $1$. Among the components that the leader did not set to $1$, the follower can arbitrarily choose $r$ components of $\bfy^p$ and $\bfy^d$, and set them to $1$. Suppose $T := \sum_{i=0}^{r-1} (Q+2^i)\bfy^p_i +         \sum_{i=0}^{r-1}  Q\bfy^d_i$, we can observe that $T$ will be between $rQ$ (if all of $\bfy^d$s are $1$s) and $rQ + 2^r-1$ (if all of $\bfy^p$s are $1$s). Now, since it is a \textsc{No} instance of \SSI, both $B-T$ and $B-T-1$ can be written as subset sum of $q_1$ to $q_k$.
But among the two, because of the pessimistic assumption, the follower will make a choice that hurts the leader the most. 
So, let $I \subseteq \{1,\dots, k\} $ such that $\sum_{i\in I} q_k = B-T$, the follower will set $\bfy_i^o=1$ for $i \in I$ and set $\bfy_s = 0$. 
This makes the leader's objective equal to $T + (B-T) = B$, completing the proof of the claim.

Thus, suppose a polynomial time algorithm exists to solve the problem in \cref{eq:QuadHardProb:b} (with oracle access to $NP$-complete problems), then we can solve \SSI in polynomial time (with oracle access to $NP$-complete problems). 
However, $\sigpt$-hardness of \SSI prevents that, indicating that \cref{eq:Sklbrg-Quad-P} is $\sigpt$-hard.
\end{proof}

The construction for the above proof can be easily extended to show that the optimistic version of the Integer Convex Quadratic Bilevel Program is also $\sigpt$-hard.
We state this formally below. 

\begin{theorem}
The optimistic version of the \em{Integer Convex Quadratic Bilevel Program} as stated in \cref{eq:Sklbrg-Quad-O} is $\sigpt$-hard.
\end{theorem}
\begin{proof}
The proof for \cref{thm:QuadIntHard} can be easily extended to show that the optimistic version of the Integer Convex Quadratic Bilevel Program is $\sigpt$-hard.
In particular, observe that the leader's objective function in \cref{eq:QuadHardProb:LeaderObj} is independent of $\bfx$, the leader's decision variables. 
Moreover, from the integrality of all the variables in the leader's objective function and the fact that all the corresponding coefficients are integers, the leader's objective function also takes only integer values.
Moreover, observe that if $M$ is chosen as an integer, then the follower's objective function in \cref{eq:QuadHardProb:FollowerObj} also takes only integer values.
Thus, suppose that  we subtract from the follower's objective a new term given by 
$
\frac{1}{2} \left(     \sum_{i=0}^{r-1} (Q+2^i)\bfy^p_i +         \sum_{i=0}^{r-1}  Q\bfy^d_i + \sum_{i=1}^{k} q_i \bfy^o_i
\right)
$, \emph{i.e., } half of the leader's objective. 
Now, among multiple equivalent solutions (all of which will have the same integer value), the fractional value subtracted now ensures that the follower will choose the solution that hurts the leader the most.
Thus, even with an optimistic assumption for the follower, a pessimistic behavior as needed in the proof of \cref{thm:QuadIntHard} can be emulated. 
The rest of the proof flows as in the proof of \cref{thm:QuadIntHard}.
\end{proof}

\subsection{Approximation algorithms}
Next, we present a few approximation algorithms for \cref{eq:Sklbrg-Quad-O} and \cref{eq:Sklbrg-Quad-P}.
Given that these problems are $\sigpt$-hard, it is unlikely that even with an oracle to solve $NP$-complete problems, these problems can be solved in polynomial time.
Thus, we first present approximation algorithms that uses an oracle to solve $NP$-complete problems, but runs in time bounded by a polynomial in the size of the input.
In other words, the algorithm we present makes only polynomially many calls to the oracle to solve $NP$-complete problems.
Moreover, in each call, the size of the $NP$-complete problem that will be solved is polynomially bounded by the size of the input to the bilevel program.
This idea is adopted in the operations research literature  too. For example \citet{Qi2024} provide an approximation algorithm for a family of integer linear bilevel programs, which are also $\sigpt$-hard. 
The algorithm runs in polynomial time \emph{after} having an access to an oracle to solve mixed-integer linear programs, an $NP$-complete problem. 
From a practical perspective, this is relevant as many commercial solvers are available for a wide range of $NP$-complete problems.
Mixed-integer linear programming solvers, SAT solvers, among others, are widely known software, which can emulate the required ``oracle''. 

\subsubsection{Polynomial-time approximation algorithm with an oracle access to solve $NP$-complete problems. }

We present \cref{alg:apx} now. 
In \cref{eq:Sklbrg-Quad-O} and \cref{eq:Sklbrg-Quad-P}, we model the leader's behaviour as follows. 
The leader makes a decision, under the assumption that the follower does not have to restrict herself to $\Z^{n_y}$. 
In other words, we solve the FRV with respect to the original problem. 
When the integer constraints are relaxed, the follower's decision is obtained as a unique solution to a system of linear equations. 
In particular, solving $\Qy \bfy = - \left (\Cy \bfx + \dy \right )$ gives the follower's unique optimal  response.
However, this means that the above equality can be added as a constraint to the leader, replacing the bilevel constraint. 
Since the follower has a unique optimal solution in the FRV, the optimistic and the pessimistic versions of the problem are equivalent.
The above gives an optimal solution $\bfx^\am$ for the leader and a possibly integer-infeasible $\hat \bfy^\am$ for the follower. 
However, now, freezing the leader's solution as $\bfx^\am$, we solve an integer-convex quadratic program, (the follower's problem) to get the optimal integer solution for the follower, denoted by $\bfy^\am$.
The algorithm is presented formally below. 
\begin{algorithm}[h]
  \caption{The Relaxed Foresight Algorithm} \label{alg:apx}
  \begin{algorithmic}[1]
    \Require{An instance of \cref{eq:Sklbrg-Quad-O} or \cref{eq:Sklbrg-Quad-P} defined by $\hx, \dx, \Qy, \Cy, \dy, \gx^i\ \forall\ i\in \{1,\ldots,m_x\}$ } 
  \Ensure {$(\bfx^\am, \bfy^\am)$ feasible to \cref{eq:Sklbrg-Quad-O} or \cref{eq:Sklbrg-Quad-P} and \emph{approximately optimal}}.
  \State Solve the FRV of \cref{eq:Sklbrg-Quad-O} or \cref{eq:Sklbrg-Quad-P} \label{alg:apx:FRV}
    to get the optimal solution $(\bfx^\am, \hat \bfy^\am)$. 
    \State Solve the convex quadratic integer program given by 
      $\min _{\bfy \in \Z^{n_y}}  \frac{1}{2} \bfy^\top \Qy \bfy + (\Cy \bfx^\am + \dy)^\top \bfy$
    to obtain the optimal solution $\bfy^\am$. If there are multiple optimal $\bfy$-s, choose the one that has the largest value of $\dx^\top \bfy$ in case of \cref{eq:Sklbrg-Quad-P}, and the smallest value in case of \cref{eq:Sklbrg-Quad-O}.
    \State \Return $\bfx^\am, \bfy^\am$.
  \end{algorithmic}
\end{algorithm}

\paragraph{Runtime analysis of \cref{alg:apx}.}
Consider the case where $\hx$ and $\dx$ are linear functions, and $\gx^i \quad \forall\ i \in \{1,\ldots,m_x\}$ is a set of $m_y$ affine constraints. In other words, $\left\{ \bfx \mid \gx^i(\bfx) \le 0 \quad \forall\ i \in \{1,\ldots,m_x\} \right\}$ is a polyhedron. 
Now, the follower's problem can be rewritten as linear equality constraints $\Qy \bfy + \Cy \bfx = -\dy$.
Thus, the FRV is a mixed-integer linear program. This is an $NP$-complete problem, which can be solved by a single call to the $NP$-complete problem oracle. 
The follower's problem involves solving an unconstrainted quadratic minimization problem over $\Z^n$. 
This is also an $NP$-complete problem, and can be solved by a single call to the $NP$-complete problem oracle. 
Only in the case where this problem has multiple optimal solutions, we need to make one more call to minimize/maximize the follower's objective among the optimal solutions to the follower's problem. 

\paragraph{Approximation guarantee.}
In the below part of the paper, we prove that the solution $(\bfx^\am,\bfy^\am)$ obtained from \cref{alg:apx} is approximate to the true optimal solution $(\bfx^\om,\bfy^\om)$ for \cref{eq:Sklbrg-Quad-O} or \cref{eq:Sklbrg-Quad-P}. 
We do so by finding an upper bound on the difference between the true optimal objective value $f^\om = \fx(\bfx^\om,\bfy^\om)$, and the objective value obtained from \cref{alg:apx}, $f^\am = \fx(\bfx^\am,\bfy^\am)$.
This gives an additive approximation guarantee for the leader's objective value.

We provide the approximation guarantee in terms of the property of the matrix $\Qy$ in the follower's problem.
We first define this property, called \emph{proximity}, which gives us a bound on the difference between the continuous and integer optima to the follower's problem.

\begin{definition}\label{def:prox}
  Let $\Qq$ be a given positive definite matrix. 
  The proximity of $\Qq$ with respect to the $\ell_p$ vector norm is denoted by $\prox[\Qq][p]$ and is the optimal objective value of the problem 
  \begin{subequations}
  \begin{align}
    \max _ {\dq \in \R^n  }\max_{ \substack{  \bfu \in \R^n \\ \bfv\in\Z^n }
    } \quad&:\quad  \norm{\bfu-\bfv}_p & \text{s.t.} \label{eq:prox:obj} \\
    \bfu \quad&\in\quad \arg \min \left \lbrace \frac{1}{2}\bfy^\top \Qq \bfy + \dq^\top \bfy  : \bfy \in \R^n \right \rbrace  \label{eq:prox:constrC} \\
    \bfv \quad&\in\quad \arg \min \left \lbrace \frac{1}{2}\bfy^\top \Qq \bfy + \dq^\top \bfy  : \bfy \in \Z^n \right \rbrace  \label{eq:prox:constr}
  \end{align}  \label{eq:prox}
  \end{subequations}
 If $p$ is not explicitly mentioned, then $p=2$ is used.
\end{definition}

Next, we state a proposition from \citet{Sankaranarayanan2024} to obtain a bound for the proximity of a positive definite matrix $\Qq$.

\begin{proposition}[\citep{Sankaranarayanan2024}] \label{thm:proxQuad}
  Given a positive definite matrix  $\Qq$ of dimension $n\times n$,  $\prox \le  \frac{\Flt{n}}{4} \sqrt{\frac{\lambda_1}{\lambda_n}}$
  where $\lambda_1$ and $\lambda_n$ are the largest and the smallest singular values of $\Qq$ and $\Flt{n}$ is a constant dependent only on the dimension $n$. 
\end{proposition}

The constant $\Flt{n}$ is known as the \emph{flatness constant}. 
It is known that if $C \subseteq \R^n$ is a convex set such that $C \cap \Z^n = \emptyset$, then there must exist a direction $\bfz$ along which $C$ should be flat, \emph{i.e., } have a width not greater than $\Flt{n}$.
The value of $\Flt{n}$ depends only on the dimension $n$ and is invariant across any convex set. 
We direct the reader to \citet{barvinok2002course} for a more elaborate discussion on the flatness constant.
Moreover, the bound $\Flt{n} \le n^{5/2}$ comes from \citet{barvinok2002course} but tighter bounds are conjectured to exist.
More recently, \citet{reis2024subspaceflatnessconjecturefaster} show that $\Flt{n} \sim O(n \log^3(2n)$.

\begin{proposition} \label{prop:proxnDiag}
     If an $n\times n$ matrix $\Qq$ is a {\em diagonal} positive definite matrix, then $\prox = \frac{\sqrt{n}}{2}$.
\end{proposition}
\begin{proof}
    Given some value of $\dq$, define $\bfu = -\Qq^{-1}\dq$.
    Define ${\bfv}$ component-wise, where $[{\bfv}]_i = \rnd([\bfu]_i)$, where $\rnd(x)$ is the integer closest to $x$, with ties broken arbitrarily. 
    First, we claim that $(\dq, \bfu, \bfv)$ together is feasible to \cref{eq:prox}.
    Clearly, by definition, $\bfu$ minimizes $f(\bfy) := \frac{1}{2}\bfy^\top \Qq \bfy + \dq^\top \bfy$ over $\R^n$. 
    So, if we show that $\bfv$ minimizes $f(\bfy)$ over $\Z^n$, the claim is proven.
    Let $\bfz=\bfu+\bfw$ minimize $f(\bfz)$ over $\Z^n$. 
    We can write $f(\bfu+\bfw) = \frac{1}{2}(\bfu+\bfw)^\top \Qq (\bfu+\bfw) + \dq^\top (\bfu+\bfw) = \frac{1}{2}\bfw^\top \Qq \bfw + \dq^\top \bfw + \frac{1}{2}\bfu^\top \Qq \bfu + \dq^\top \bfu + \bfw^\top\Qq \bfu = \frac{1}{2}\bfw^\top \Qq \bfw  + \frac{1}{2}\bfu^\top \Qq \bfu + \dq^\top \bfu + (\dq+\Qq\bfu)^\top \bfw$. However, since $\Qq\bfu + \dq = 0$, the last term vanishes, making $f(\bfu+\bfw) = \frac{1}{2}\bfw^\top \Qq \bfw  + \frac{1}{2}\bfu^\top \Qq \bfu + \dq^\top \bfu$. We want to choose $\bfw$ to minimize the above such that $\bfu+\bfw \in \Z^n$. However, since $\bfu$ is given, the last term is a constant, hence it is equivalent to minimizing $\frac{1}{2}\bfw^\top \Qq \bfw $ such that $\bfu+\bfw = \Z^n$. However, we know that $\Qq$ is diagonal. This means, $\frac{1}{2}\bfw^\top \Qq \bfw = \frac{1}{2} \sum_{i=1}^n[\Qq]_{ii}\bfw_i^2 $. However, the above is separable, implying that it is minimized if each $\bfw_i$ individually takes as small a value as possible, so that $\bfu+\bfw \in \Z^n$. But this precisely is equivalent to saying that the rounded version of $\bfu$ is the integer optimum. This proves the claim.

    Now, to prove the proposition, we observe that the largest norm possible for $\bfw$ is when each $\bfw_i$ is as large as possible to be rounded off, which is $0.5$. 
    This has an $\ell_2$ norm of $\sqrt{n}/2$, proving that $\prox \le \frac{\sqrt{n}}{2}$.
    The equality follows from the case, when the continuous minimum is at $0.5 \mathbf{e}$.
\end{proof}

    A function $f:D\to\R$ is said to be $\ell_p$-Lipschitz continuous with a Lipschitz constant $L_p$ over some $D\subseteq \R^n$, if $\vert{f(\bfx)-f(\bfy)}\vert \le L_p \norm{\bfx-\bfy}_p$ for every $\bfx, \bfy \in D$.
  Assuming that $\dx$ is $\ell_2$-Lipschitz continuous, we present the main approximation theorem.

\begin{theorem}[Ex-ante and ex-post guarantees in $\ell_2$ norm] \label{thm:eapGty}
  Suppose $\dx(\cdot)$ is $\ell_2$-Lipschitz continuous over the feasible set of the leader (\emph{i.e.,} the region defined by $\gx^i(\bfx) \le 0 \quad \forall\ i \in \{1,\ldots,m_x\}$) with a Lipschitz constant $L_2$.
  Let the optimal objective obtained by the \emph{leader} in \cref{eq:Sklbrg-Quad-O} or \cref{eq:Sklbrg-Quad-P} be $f^\om$. Then, 
    \begin{subequations}
        \begin{align}
            f^\am \quad&\le\quad f^\om + 2 L_2\prox[\Qy]
  \le f^\om + 2 L_2
  \frac{\Flt{n}}{4} \sqrt{\frac{\lambda_1}{\lambda_n}}
  \label{thm:eaGty} \\ 
            f^\am \quad&\le\quad f^\om +  L_2 (\prox[\Qy] + \norm{\bfy^\am - \hat \bfy^\am}_2), \label{thm:epGty}
        \end{align}
    \end{subequations}
  where $f^\am$ is the objective function value of the solution obtained from \cref{alg:apx}.
 Here, \cref{thm:eaGty} is the ex-ante guarantee and \cref{thm:epGty} is the ex-post guarantee in the $\ell_2$ norm.
 Further, if $\Qy$ is a diagonal matrix, then in \cref{thm:eaGty}, $f^\am$ is at most $f^\om +2L_2\sqrt{n_y}$ in the worst case.
\end{theorem}

By providing a finite bound for the difference between the objective value from the true solution and the one obtained from \cref{alg:apx}, this theorem establishes that the solution obtained from \cref{alg:apx} approximates the true optimal solution.
Now we proceed to prove \cref{thm:eapGty}. 

\begin{proof}[Proof of \cref{thm:eapGty}.]

Suppose $(\bfx^\om, \bfy^\om)$ is the (unknown) optimal solution to \cref{eq:Sklbrg-Quad-O} or \cref{eq:Sklbrg-Quad-P}. 
  Let $\hat \bfy^\om$ be the continuous minimiser of the follower's problem, given the leader's optimal decision $\bfx^\om$. 
 Given a leader's decision, the continuous relaxation of the follower's problem has a unique solution since $\Qy$ is assumed to be positive definite.
 Since there is a unique continuous minimiser, the FRV is same for both the pessimistic and the optimistic version of the problem. 
  
 Now, we apply \cref{thm:proxQuad} to the follower's problem.
 By \cref{thm:proxQuad}, we know that $\norm{\hat \bfy^\om - \bfy^\om}_2 \le \prox[\Qy][2]$.
    This inequality holds irrespective of whether $\bfy^\om$ is chosen pessimistically or optimistically as \cref{thm:proxQuad} holds for every choice of integer and continuous minimisers. 
    Furthermore, since $\dx (\cdot)$ is Lipschitz continuous with constant $L_2$, we know that 
    $ \left \vert \dx(\hat \bfy^\om) - \dx(\bfy^\om) \right \vert \le L_2 \norm{\hat \bfy^\om - \bfy^\om}_2$, implying that
    $\dx (\hat \bfy^\om) \leq \dx (\bfy^\om) + L_2\prox[\Qy][2]$. 
    
  Now, since $(\bfx^\am,\hat \bfy^\am)$ is an optimal solution to the FRV while $(\bfx^\om, \hat \bfy^\om)$ is feasible to the FRV, we have
  \begin{subequations}
  \begin{align}
    \fx(\bfx^\am, \hat \bfy^\am) 
      \quad&=\quad 
       \hx (\bfx^\am) + \dx (\hat \bfy ^\am) \\
       \quad&\leq\quad \label{ineq:quad-hats}
       \hx (\bfx^\om) + \dx (\hat \bfy^\om) \\
       \quad&\leq\quad
       \hx (\bfx^\om) + \dx (\bfy^\om) + L_2\norm{\hat \bfy^\om - \bfy^\om}_2
      \\
      \quad&\leq\quad
      f^\om + L_2\prox[\Qy][2],
    \end{align}\label{eq:t1-1}
  where $f^\om = \hx (\bfx^\om) + \dx ^\top \bfy ^\om $. $f^\om$ is the value of the objective function at the optimal solution $(\bfx^\om,\bfy^\om)$.  
  \end{subequations}
However, $(\bfx^\am, \hat \bfy^\am)$ is not necessarily feasible to \cref{eq:Sklbrg-Quad-P} or \cref{eq:Sklbrg-Quad-O}, as $\hat \bfy^\am$ may not be integral. 
$\bfy^\am$ is an integer optimum to the follower given $\bfx^\am$, hence $(\bfx^\am, \bfy^\am)$ is feasible. 

Given $(\bfx^\am$, $\hat \bfy^\am)$ is the continuous optimum to the lower-level problem and $\bfy^\am$ is an integer optimum to the lower-level problem, keeping the optimistic/pessimistic assumptions in mind, we know from \cref{thm:proxQuad} that $\norm{\bfy^\am - \hat \bfy^\am} \le \prox[\Qy]$. We also know that $\dx (\bfy^\am) \leq \dx (\hat \bfy^\am) + L\prox[\Qy]$. We use these results to provide a bound for $\fx (\bfx^\am, \bfy^\am) - \fx(\bfx^\am, \hat \bfy^\am)$, the latter of which has been proved to approximate $\fx (\bfx^\om,\bfy^\om)$ in \cref{eq:t1-1}.  
 We note that

\begin{subequations}
  \begin{align}
    \fx(\bfx^\am, \bfy^\am) \quad&=\quad 
       \hx (\bfx^\am) + \dx (\bfy ^\am) \\
       \quad&\leq\quad
       \hx (\bfx^\am) + \dx (\hat \bfy^\am) + L_2\norm{\bfy^\am - \hat \bfy^\am}_2 
       \\
       \quad&\leq\quad
       f^\om + 2L_2\prox[\Qy][2],
  \end{align} \label{eq:t1-2}
  where the first inequality is from $\dx(\cdot)$ being Lipschitz continuous, and the last inequality follows from \cref{eq:t1-1}. 
 \end{subequations}
 This proves the first inequality in \cref{thm:eaGty}. The second inequality in \cref{thm:eaGty} follows from the bound in \cref{thm:proxQuad}. 

If $\bfy^\am$ and $\hat \bfy^\am$ from the algorithm are known, then from \cref{eq:t1-1},
\begin{subequations}
    \begin{align}
    \fx(\bfx^\am, \bfy^\am) \quad&=\quad 
       \hx (\bfx^\am) + \dx (\bfy ^\am) \\
       \quad&\leq\quad
       \hx (\bfx^\am) + \dx (\hat \bfy^\am) + L_2\norm{\bfy^\am - \hat \bfy^\am}_2 
       \\
       \quad&\leq\quad
       f^\om + L_2\prox[\Qy][2]
            + L_2\norm{\bfy^\am - \hat \bfy^\am}_2,   
    \end{align} 
    where the last inequality follows from \cref{eq:t1-1}.
\end{subequations}
This proves the inequality in \cref{thm:epGty}.
This is tighter bound for $f^\om - f^\am$, given that the approximate solution has already been obtained from \cref{alg:apx}. 

Finally, from \cref{prop:proxnDiag}, if $\Qy$ is a diagonal matrix, then $\prox[\Qy] \le \sqrt{n_y}$.
Substituting this in \cref{thm:eaGty} gives the bound $f^\am \le f^\om + 2L_2\sqrt{n_y}$.
\end{proof}

\cref{thm:eapGty} shows that the solution obtained from \cref{alg:apx} approximates the true solution, since the objective value from both solutions differs by a bounded quantity.
If the leader has a linear objective function in \cref{eq:Sklbrg-Quad-P}, we can use a known norm in place of the Lipschitz constant $L_2$ in \cref{thm:eapGty}.
\begin{theorem} \label{cor:eapGty} 
    Suppose the function $\dx(\bfy) := \dx^\top \bfy$. 
    Then, the following bounds hold for the objective $f^\am$ obtained using \cref{alg:apx}:
    \begin{subequations}
        \begin{align}
            f^\am \quad&\le\quad f^\om  
+\frac{\Flt{n_y}\sqrt{\lambda_1}}{2 \sqrt{2}\norm{\Rq^{-\top} \dx}} \dx^\top  \Qy ^{-1} \dx,
  \label{cor:eaGty} 
        \end{align}
    \end{subequations}
 where $f^\om$ is the true optimal objective value, and $\Rq$ is a matrix such that $\Qq  = \Rq^\top \Rq$, \emph{i.e.,} $\Rq$ is  the Cholesky decomposition of $\Qq$..
\end{theorem}

Before we prove the theorem,  we state the following lemmata. The proof of these lemmata are provided in the online supplement. 
\begin{lemma}\label{thm:circVal}
Let $\Qq$ be an $n\times n$ positive definite matrix. 
Then, the optimal objective value of the problem 
$
\max\limits_{\bfx} \left\{ \pq^\top \bfx \mid \bfx^\top \Qq \bfx \le \gamma \right\}
$ is 
$= \frac{\sqrt{\gamma}}{\norm{\Rq^{-\top} \pq}} \pq^\top  \Qq ^{-1} \pq$, where $\Rq$ is a matrix such that $\Qq  = \Rq^\top \Rq$, \emph{i.e.,} $\Rq$ is  the Cholesky decomposition of $\Qq$.
\end{lemma}

\begin{proof}[Proof of \cref{thm:circVal}]
    Since $\Qq$ is positive definite, there exists an invertible matrix $\Rq$ such that $\Qq = \Rq^\top \Rq$. 
    Thus, the problem can be rewritten as 
    $
\max\limits_{\bfx} \left\{ \pq^\top \bfx \mid \bfx^\top \Rq^\top \Rq \bfx \le \gamma \right\}
= 
\max\limits_{\bfy} \left\{ \pq^\top \Rq^{-1} \bfy \mid \bfy^\top \bfy \le \gamma \right\}
$, after the substitution $\Rq \bfx = \bfy$. 
This is now a problem of finding a vector $\bfy$ whose $\ell_2$ norm is at most $\sqrt{\gamma}$, and whose dot product with $\Rq^{-\top} \pq$ is maximised.
For a fixed norm, the dot product is maximized by a parallel vector, appropriately scaled. 
So, the maximum is achieved by $\bfy = \sqrt{\gamma} \Rq^{-\top}\pq / \norm{\Rq^{-\top}\pq}_2$. 
Substituting back, the maximum value is 
$
\pq^\top \Rq^{-1} \left( 
 \sqrt{\gamma} \Rq^{-\top}\pq / \norm{\Rq^{-\top}\pq}_2
\right)
$
$
= \frac{\sqrt{\gamma}}{\norm{\Rq^{-\top} \pq}} \pq^\top  \left( \Rq^{\top} \Rq \right ) ^{-1} \pq
= \frac{\sqrt{\gamma}}{\norm{\Rq^{-\top} \pq}} \pq^\top  \Qq ^{-1} \pq
$.
\end{proof}

\begin{lemma} \label{thm:proxVal}
  Let $\Qq$ be an $n\times n$ positive definite matrix. Then, $\max_{\pq\in\R^n} \max _{\bfu\in\R^n, \bfv\in\Z^n} \pq ^\top \left( \bfu - \bfv \right) $ subject to \cref{eq:prox:constrC,eq:prox:constr} is at most 
  $= \frac{\Flt{n}\sqrt{\lambda_1}}{4 \sqrt{2}\norm{\Rq^{-\top} \pq}} \pq^\top  \Qq ^{-1} \pq$, where
  $\lambda_1$ is the largest eigen value of $\Qq$, 
  $\Flt{n}$ is the flatness constant,  and 
  $\Rq$ is a matrix such that $\Qq  = \Rq^\top \Rq$, \emph{i.e.,} $\Rq$ is  the Cholesky decomposition of $\Qq$.
\end{lemma}
\begin{proof}[Proof of \cref{thm:proxVal}]
  From \citet[Theorem 3 (ii)]{Sankaranarayanan2024}, we know that for a quadratic function
  $\frac{1}{2} \bfx^\top \Qq \bfx $
  , whose minimum is at the origin, the integer minima are located within the ellipse given by $\frac{1}{2} \bfx^\top \Qq \bfx \le \gamma$, where $\gamma=\frac{\lambda_1 \Flt{n}^2}{32}$.
  But from \cref{thm:circVal}, the maximum value of the $\pq^\top \bfx$ over this ellipse is $ f^\star = \frac{\sqrt{\gamma}}{\norm{\Rq^{-\top} \pq}} \pq^\top  \Qq ^{-1} \pq$. 
  Thus the largest value $\pq^\top \bfv$ can take is $f^\star$.
  Since the continuous minimum is the origin, the only value of $\pq^\top \bfu = 0$. Thus the maximum difference as needed is indeed $f^\star$.
  However, the choice that the continuous minimum is at the origin is arbitrary.
  Translating the continuous minimum also translates the maximum distance from the integer minimum, but does not affect the difference between the two.
\end{proof}
\noindent With the above two lemmata, we can prove \cref{cor:eapGty}.

\begin{proof}[Proof of \cref{cor:eapGty}]
Using notations analogous to that in the proof of \cref{thm:eapGty}, we have
  \begin{subequations}
  \begin{align}
    \fx(\bfx^\am, \hat \bfy^\am) 
      \quad&=\quad 
       \hx (\bfx^\am) + \dx ^\top \hat \bfy ^\am \\
       \quad&\leq\quad 
       \hx (\bfx^\om) + \dx ^\top \hat \bfy^\om \\
       \quad&=\quad
       \hx (\bfx^\om) + \dx^\top \left( \hat \bfy^\om - \bfy^\om  \right) + \dx^\top \bfy^\om
      \\
      \quad&\leq\quad 
       \hx (\bfx^\om) + 
   \frac{\Flt{n}\sqrt{\lambda_1}}{4 \sqrt{2}\norm{\Rq^{-\top} \dx}} \dx^\top  \Qy ^{-1} \dx
       + \dx^\top \bfy^\om
      \\
      \quad&=\quad
      f^\om 
   +\frac{\Flt{n}\sqrt{\lambda_1}}{4 \sqrt{2}\norm{\Rq^{-\top} \dx}} \dx^\top  \Qy ^{-1} \dx,
    \end{align}
  \end{subequations}
  where $\Rq$ is the Cholesky decomposition of $\Qy$. \emph{i.e., } $\Qy = \Rq^\top \Rq$, and $\lambda_1$ is the largest eigenvalue of $\Qy$.
  Here, the first inquality is due to the optimality of $(\bfx^\am, \hat \bfy^\am)$ for the FRV. 
  The equality in the following line is because of adding and subtracting $\dx^\top \bfy^\om$ terms. 
  The following inequality is due to \cref{thm:proxVal}. 
  The equality in the end is the definition of $f^\om$.
  Now, we know that 
\begin{subequations}
  \begin{align}
    \fx(\bfx^\am, \bfy^\am) \quad&=\quad 
       \hx (\bfx^\am) + \dx ^\top \bfy^\am \\
       \quad&=\quad
       \hx (\bfx^\am) + \dx^\top\left ( \bfy^\am -  \hat \bfy^\am  \right ) + \dx ^\top \hat \bfy^\am
       \\
       \quad&\leq\quad
       f^\om 
+\frac{\Flt{n}\sqrt{\lambda_1}}{2 \sqrt{2}\norm{\Rq^{-\top} \dx}} \dx^\top  \Qy ^{-1} \dx,
  \end{align} 
  where the inequality is due to the application of  \cref{thm:proxVal}. 
 \end{subequations}
\end{proof}
We note that a linear function $\dx^\top \bfy$ is indeed Lipschitz continuous with a constant $\norm{\dx}_2$.
Thus, a trivial bound for the case in \cref{cor:eapGty} is obtained by applying \cref{thm:eaGty} with $L_2 = \norm{\dx}_2$.
However, the bound provided in \cref{cor:eapGty} is tighter than the trivial bound, but is only valid in the special case.

\subsubsection{Polynomial-time approximation algorithm without any oracle access.}
Now, we consider purely polynomial-time algorithms that do not require any oracle access to solve $NP$-complete problems. 
We note that, unless $NP = \sigpt$, there does not even exist a certificate and a polynomial-time algorithm to verify the feasibility of a given solution $(\bfx, \bfy)$ to \cref{eq:Sklbrg-Quad-P} or \cref{eq:Sklbrg-Quad-O}.
So, in this part of the manuscript, we provide only a leader's decision $\bfx$. 
In many situations, the fundamental difficulty in a bilevel program is to identify the optimal leader's decision. 
Thus, we provide the following theorem, which provides approximation guarantees on when the FRV can be solved approximately. 

\begin{theorem}[Polynomial-time approximation algorithm] \label{thm:apx}
  Let an instance of \cref{eq:Sklbrg-Quad-O} or \cref{eq:Sklbrg-Quad-P} be given where $\dx$ is an $\ell_2$-Lipschitz continuous function over the feasible set of the leader. 
  Let $\Qy$ be a positive definite matrix. 
  Further, suppose that the FRV is solved approximately to get a solution $(\bfx^\am, \hat \bfy^\am)$ such that if $\hat f^\om$ is the optimal objective of the FRV, the obtained solution has an objective at most $\alpha \hat f^\om + \beta$ for some $\alpha \ge 1$ and $\beta \ge 0$. 
  Then, if the leader plays the strategy $\bfx^\am$ and the follower responds optimally in the given instance of \cref{eq:Sklbrg-Quad-O} or \cref{eq:Sklbrg-Quad-P}, the objective value of the leader is at most $\alpha f^\om + \beta + (\alpha+1)L_2\prox[\Qy]$ where $f^\om$ is the optimal objective value of the given instance.
\end{theorem}

The above theorem provides the guarantees on the original problem under the case where the FRV can be solved approximately in polynomial time. 

\begin{proof}
  Here, let $\bfx^\am$ be the leader's decision in the FRV solved approximately. 
  Let $\bfy^\am$ be the follower's optimal response (with the integer constraints). 
  Now, we can write 
  \begin{subequations}
    \begin{align}
      \fx(\bfx^\am, \bfy^\am) \quad&=\quad \hx(\bfx^\am) + \dx(\bfy^\am) \\
      \quad&=\quad \hx(\bfx^\am) + \dx(\hat \bfy^\am + \bfy^\am - \hat \bfy^\am) \\
      \quad&\leq\quad \hx(\bfx^\am) + \dx(\hat \bfy^\am) + L_2\norm{\bfy^\am - \hat \bfy^\am}_2 \\
      \quad&\leq\quad \hx(\bfx^\am) + \dx(\hat \bfy^\am) + L_2\prox[\Qy] \\
      \quad&\leq\quad \alpha \left( \hx (\bfx^\om) + \dx(\hat \bfy^\om) \right) + \beta + L_2\prox[\Qy] \\
      \quad&=\quad \alpha \left( \hx(\bfx^\om) + \dx (\bfy^\om + \hat \bfy^\om - \bfy^\om ) \right) + \beta + L_2 \prox[\Qy] \label{eq:apxapx:a} \\
    \quad&\leq\quad\alpha \left( \hx(\bfx^\om) + \dx (\bfy^\om) + L_2\norm{\hat \bfy^\om - \bfy^\om } \right) + \beta + L_2 \prox[\Qy] \\
    \quad&\leq\quad\alpha \left( \hx(\bfx^\om) + \dx (\bfy^\om) + L_2\prox[\Qy] \right) + \beta + L_2 \prox[\Qy] \\
    \quad&=\quad \alpha f^\om + \beta + (\alpha+1)L_2\prox[\Qy]
    \end{align}
  \end{subequations}
  Here, the first two equalities follow from the definition and by adding and subtracting $\hat \bfy^\am$. 
  The inequality in the following line is due to the Lipschitz continuity of $\dx$.
  The next inequality is due to \cref{thm:proxQuad}.
  The next equality is due to the assunmption that $(\bfx^\am, \hat\bfy^\am)$ solves the FRV approximately.
  If the inequality holds for the optimal solution, it also holds for $(\bfx^\om, \hat\bfy^\om)$ where $\bfx^\om$ is the true optimal solution for the given instance of \cref{eq:Sklbrg-Quad-O} or \cref{eq:Sklbrg-Quad-P} and $\hat \bfy^\om$ is the continuous minimiser of the follower's problem given the leader's decision $\bfx^\om$.
  The equality, following in \cref{eq:apxapx:a} is because we are adding and subtracting $\bfy^\om$.
  The inequality in the next line is due to the Lipschitz continuity of $\dx$.
  The following inequality is due to \cref{thm:proxQuad}.
  Finally, the equality in the last line follows from the definition of $f^\om$ and rearranging the terms. 
\end{proof}

\section{Extensions to bilevel programs with an integer-linear follower}
In this section, we extend our proximity-based results to
settings where 
\begin{enumerate}
    \item given a leader's decision, the follower solves an integer linear program;
    \item given any $\bfx$ satisfying the constraints of the leader, the follower's problem is feasible.
    \item the leader's and the follower's objectives are misaligned, {\em i.e., } given the form in \cref{eq:Sklbrg-Gen-O}, we require that $\dx(\cdot) = - \dy(\cdot)$.
\end{enumerate}
For preciseness, we provide the form of the problem as below. 
\begin{align}
   \min_{\bfx\in \Z^{n_x}, \bfy \in \Z^{n_y}} \quad&:\quad \hx (\bfx) + \dq^\top \bfy &\text{s.t.} \nonumber \\
   \gx^i (\bfx) \quad&\leq\quad \sx \quad \forall\ i \in \{1,\ldots,m_x\} \nonumber \\
   \bfy \quad&\in\quad \arg \max_{\bfy \in \Z^{n_y}} \left\{ \dq^\top \bfy \,:\, \gy (\bfx) + \Dy \bfy \,\leq\, \sy \right\} \tag{ICB-Lin} \label{eq:Sklbrg-Lin}
 \end{align}
 with the condition that if for some $\bfx = \bar\bfx$, $\gx(\bar\bfx) \le 0$, then there exists a $\bar\bfy$ such that $\gy(\bar \bfx) + \Dy \bar \bfy \le 0$. We also assume that $\gy(\bfx)$ takes integer values for every feasible $\bfx$. We note that the condition that $\gy(\bfx)$ is integer valued for every feasible $\bfx$ is satisfied in a fair number of standard settings. 
For example, if the leader's variables $\bfx$ are constrained to be integers, and the follower's linear constraints are of the form $\Cy \bfx + \Dy \bfy \le \bq$, where $\Cy$ is a matrix with integer constraints,  then this condition is satisfied. 
 
We note that the follower is {\em maximizing} $\dq^\top \bfy$, while the leader is minimizing their objective, with $\dq^\top \bfy$, being the only term in the leader's objective, that has the follower's variables. 
Furthermore, we note that the above property ensures that the solution for the pessimistic and optimistic versions of the  bilevel program always match. {\em i.e., } any feasible solution to the optimistic version is also feasible to the pessimistic version; conversely, any feasible solution to the pessimistic version is also feasible to the optimistic version. Thus, their optimal solutions match too, and it becomes essentially equivalent to solve either of the two problems. 

The results we state in this section also extend naturally, and some times in stronger form, if $\bfy$ is constrained to be mixed-integer as opposed to pure-integer. However, for conciseness, we present the results only for the pure-integer setting. 
Moreover, one can readily observe that if $\hx(\bfx) = \hx^\top \bfx$, $\gx(\bfx) := A\bfx \le \mathbf{g}$, {\em i.e.,} the constraints $\gx(\bfx) \le 0$ define a polyhedron, and $\gy(\bfx) = \Cx \bfx + \bq$, we retrieve the classical integer-bilevel linear program, but with misaligned objectives. 

\begin{theorem}
    Given an instance of \cref{eq:Sklbrg-Lin}, and some $\gamma \in \R$, deciding if the optimal solution to the instance is at most $\gamma$ is $\sigpt$-hard. 
\end{theorem}
The DNeg version of the bilevel knapsack problem discussed in \citet[Sec 3.3]{caprara_study_2014} is already in the form of \cref{eq:Sklbrg-Lin}.
{\em i.e.,} it satisfies all the three conditions stated in the beginning of the section. 
However, the version of bilevel knapsack is $\sigpt$, giving a natural proof to the above theorem. 

Now, we provide the approximate algorithm, analogous to \cref{alg:apx}. 
\begin{algorithm}[h]
  \caption{The Relaxed Foresight Algorithm for Integer Linear Follower} \label{alg:apx:lin}
  \begin{algorithmic}[1]
    \Require{An instance of \cref{eq:Sklbrg-Lin} defined by $\hx, \dq, \gx^i\ \forall\ i\in \{1,\ldots,m_x\}, \gy^i\ \forall\ i\in \{1,\ldots,m_y\}, \Dy$ } 
    \Ensure {$(\bfx^\am, \bfy^\am)$ feasible to \cref{eq:Sklbrg-Lin} and \emph{approximately optimal}} 
    \State Solve the FRV of \cref{eq:Sklbrg-Lin}
    to get the optimal solution $(\bfx^\am, \hat \bfy^\am)$ 
    \State Solve the integer linear program given by 
      $\max _{\bfy \in \Z^{n_y}} \dq^\top \bfy $ subject to $\gy(\bfx) + \Dy \bfy \le 0$
    to obtain the optimal solution $\bfy^\am$ \label{alg:apx:lin:foll}
    \State \Return $\bfx^\am, \bfy^\am$ 
  \end{algorithmic}
\end{algorithm}

We analyze the above algorithm using proximity results for integer linear programs, analogous to the analysis in \cref{sec:Quad}.
We provide separate results for two subcases, one which is more useful if there are more variables than constraints for the follower, and another which is more useful if there are more constraints than variables for the follower. 

\subsection{Guarantees for when there are few variables and many constraints}

We first look into a result defining the proximity measure for integer linear programs. This result will be used as a bound for the distance between the continuous and integer optima to the follower's problem.

\begin{proposition}[Proximity \citep{cook1986}] \label{thm:Cook}
 Let $\Dy$ be an integral matrix such that each subdeterminant is at most $\Delta (\Dy)$ in absolute value, $\dy$ and $b$ be vectors such that $\Dy \bfy \le \bq$ has an integral solution and $\min\{\dy^\top \bfy : \Dy \bfy \le \bq\}$ exists. Suppose that both
    \begin{subequations}
    \begin{gather} 
\min \left \{ \dy^\top \bfy \mid 
\Dy \bfy \le \bq  \right \} \label{cook-a}
\\
\min \left \{ \dy^\top \bfy \mid \Dy \bfy \le \bq ; \bfy \in \Z^n\right \} \label{cook-b}
    \end{gather}
    \end{subequations}
    are finite. Then,
    \begin{subequations}
    \begin{align}
        \forall\ \hat \bfy\ \text{optimal to \cref{cook-a}}\  &\exists\ \tilde{\bfy}\ \text{optimal to \cref{cook-b} s.t.}\  \norm{\hat \bfy - \tilde{\bfy}}_\infty \le n\Delta (\Dy) & \text{and}\\
        \forall\ \tilde{\bfy}\ \text{optimal to \cref{cook-b}}\ &\exists\ \hat \bfy\ \text{optimal to \cref{cook-a} s.t.}\ \norm{\hat \bfy - \tilde{\bfy}}_\infty \le n\Delta (\Dy).
    \end{align}
    \end{subequations}   
 \end{proposition}
We note that the bound as provided above, does not depend on the RHS $\bq$ of the constraint. In other words, if the RHS $\bq$ varies, the same bound continues to hold and remains valid. 
We also note that, unlike the result in \cref{thm:proxQuad}, only guarantees the existence of {\em an} integer optimal solution that is sufficiently proximal, and there may still be other integral optima that are farther away from $\hat \bfy$.

The existence of an integral solution $\tilde \bfy$ to the follower's problem in \cref{eq:Sklbrg-Lin} that is proximal to the solution of the FRV $\hat \bfy$ lets us develop a result to prove that the solution $(\bfx^\am,\bfy^\am)$ obtained from \cref{alg:apx:lin} approximates the true solution $(\bfx^\om,\bfy^\om)$. 

\begin{theorem} \label{thm:eapGty-Lnew}
  Let the optimal objective obtained by the \emph{leader} in \cref{eq:Sklbrg-Lin} be $f^\om$. 
 Then, the following bound holds for the objective $f^\am$ obtained using \cref{alg:apx:lin}:
 \begin{subequations}
 \begin{align}
 f^\am \quad&\le\quad f^\om + 2\norm{\dq}_1 n_y \Delta (\Dy)  
 \end{align}
 \end{subequations}
\end{theorem}
\begin{proof}
  Let $(\bfx^\om,\bfy^\om)$ be a true optimal solution for \cref{eq:Sklbrg-Lin}. Let $\hat \bfy^\om$ be the continuous minimizer of the follower's problem, given the leader's decision $\bfx^\om$. 
  We know from \cref{thm:Cook} that there exists $ \tilde{\bfy}^\om$ feasible and optimal to the (integer-constrained) follower problem given the leader's decision $\bfx^\om$, such that $\norm{\tilde{\bfy}^\om - \hat \bfy^\om}_\infty \le n \Delta (\Dy)$.

 We note that, while $\tilde \bfy^\om$ is optimal to the follower's problem (given the leader's decision), it is only {\em an} optimal solution, and other optimal solutions to the follower, which are farther away could exist. 
 Let $\tilde \bfy^\om_1, \ldots, \tilde \bfy^\om_k$ be other such optimal solutions. 
 Since they are all optimal, the value of the follower's objective, $\dq^\top \tilde \bfy^\om_i$ is the same for all $i=1,\ldots, k$ and is the same as $\dq^\top\tilde \bfy^\om$. 
 However, due to misaligned objectives between the leader and the follower, the leader's objective is also the same at all these points. 
  
  Now, by the fact that $(\bfx^\am, \hat \bfy^\am)$ are obtained as optimal solutions to the corresponding FRV, we have
  \begin{subequations}
  \begin{align}
       \hx (\bfx^\am) + \dq^\top  \hat \bfy ^\am
       \quad&\leq\quad
       \hx (\bfx^\om) + \dq^\top \hat \bfy^\om \\
            \quad&=\quad
       \hx (\bfx^\om) + \dq^\top\tilde{\bfy}^\om + \dq^\top (\hat \bfy^\om 
    - \tilde {\bfy}^\om
    )\\
       \quad&\leq\quad 
       \hx (\bfx^\om) + \dq^\top\tilde{\bfy}^\om + \norm{\dq}_1 \norm{\hat \bfy^\om - \tilde{\bfy}^\om}_\infty \\
       \quad&\leq\quad
       \hx (\bfx^\om) + \dq^\top \tilde {\bfy}^\om + \norm{\dq}_1 n_y \Delta (\Dy) \\
            \quad&=\quad
       \hx (\bfx^\om) + \dq^\top  {\bfy}^\om + \norm{\dq}_1 n_y \Delta (\Dy) \\
      \quad&=\quad
      f^\om + \norm{\dq}_1 n_y \Delta (\Dy)
    \end{align} \label{eq:t2new-1}
  \end{subequations}
 where $f^\om = \hx (\bfx^\om) + \dq^\top \bfy ^\om $ is the optimum objective value of \cref{eq:Sklbrg-Lin}. 
 
The first inequality in \cref{eq:t2new-1} is due to the fact that $(\bfx^\am, \hat \bfy^\am)$ is feasible and optimal to the FRV, while $(\bfx^\om, \hat \bfy^\om)$ is feasible to the FRV.
The following equality is obtained by adding and subtracting $\tilde \bfy^\om$.
The following equality follows from the H\"older's inequality 
{\em i.e.,} $\left \vert{\dq^\top (\hat \bfy^\om 
    - \tilde {\bfy}^\om
    )} \right \vert \le \norm{\dq}_1 \norm{\hat{\bfy}^\om - \tilde \bfy^\om}_\infty$.
The next inequality follows from \cref{thm:Cook}.
The equality in the penultimate line is because $\bfy^\om$ was the assumed optimum, and it does have the same objective value as $\tilde \bfy^\om$. Finally, the last 
equality is due to the fact that $(\bfx^\am, \tilde \bfy^\am)$ is an integer feasible and optimal solution. 

However, $(\bfx^\am, \hat \bfy^\am)$ is not necessarily feasible to \cref{eq:Sklbrg-Lin} since $\hat \bfy^\am$ may not satisfy the integrality requirements. The corresponding feasible point is $(\bfx^\am, \bfy^\am)$. 
Given $\bfx^\am$, $\hat \bfy^\am$ is the continuous optimum to the lower-level problem, we know from \cref{thm:Cook} that there exists $\tilde{\bfy}^\am$, an integer optimum to the lower-level problem such that $\norm{\hat \bfy^\am - \tilde{\bfy}^\am}_\infty \le n \Delta (\Dy)$.
Thus,
\begin{subequations}
\begin{align}
       \hx (\bfx^\am) + \dq ^\top \bfy ^\am
       \quad&=\quad
       \hx (\bfx^\am) + \dq ^\top \hat \bfy^\am + \dq ^\top \bfy ^\am - \dq ^\top \hat \bfy^\am 
       \\
            \quad&\le\quad
            f^\om + \norm{\dq}_1 n_y \Delta (\Dy) + \dq ^\top \bfy ^\am - \dq ^\top\tilde \bfy^\am  + \dq ^\top(\tilde \bfy^\am -  \hat \bfy^\am) 
            \\
            \quad&=\quad
            f^\om + \norm{\dq}_1 n_y \Delta (\Dy)  + \dq ^\top(\tilde \bfy^\am -  \hat \bfy^\am) 
            \\
       \quad&\leq\quad
            f^\om + \norm{\dq}_1 n_y \Delta (\Dy)  + \norm{\dq}_1 \norm{\tilde \bfy^\am  - \hat \bfy^\am}_\infty
            \\
       \quad&\leq\quad
            f^\om + \norm{\dq}_1 n_y \Delta (\Dy)  + \norm{\dq}_1 n_y \Delta(\Dy)
            \\
       \quad&=\quad
            f^\om + 2\norm{\dq}_1n_y \Delta (\Dy) 
  \end{align} \label{eq:t2-2}
 \end{subequations}
where the first inequality follows from \cref{eq:t2new-1} along with adding and subtracting $\dq^\top \tilde \bfy^\am$.
The equality in the following line is because $\tilde \bfy^\am$ as well as $\bfy^\am$ are follower optimal, and hence have the same follower objective value.
The inequality in the following line is due to the H\"older's inequality.
The next inequality follows from \cref{thm:Cook}.  
\end{proof}
\cref{thm:eapGty-Lnew} shows that the solution obtained from \cref{alg:apx:lin} approximates the true solution, as the objective value from both solutions differs by a bounded amount.

\subsection{Guarantees for when there are many variables and few constraints}
We start this subsection, presenting a theorem from \citet{eisenbrand2018proximity}, which provides a bound on the $\ell_1$ norm distance between the continuous and integer solutions to an integer linear program.
\begin{proposition}[$\ell_1$ norm proximity between relaxed and integer solutions \citep{eisenbrand2018proximity}] \label{thm:E-W}
  Let $\Dy$ and $b$ be integral matrices such that $\delta (\Dy)$ is the upper bound on the absolute values of entries in $\Dy$. Suppose that both
    \begin{subequations}
    \begin{gather} 
\min \left \{ \py^\top \bfy \mid 
\Dy \bfy \le \bq  \right \} \label{one-a}
\\
\min \left \{ \py^\top \bfy \mid \Dy \bfy \le \bq; \bfy \in \Z^n\right \} \label{one-b}
    \end{gather} \label{one-ab}
    \end{subequations} 
    are finite with optimal solutions $\bfy^\am$ and $\bfy^\om$ respectively. Then,
    \begin{subequations}
        \begin{align}
        \forall\ \hat \bfy\ \text{optimal to \cref{one-a}}\ \exists\ \tilde \bfy\ \text{optimal to \cref{one-b} s.t.}\  \norm{\hat \bfy - \tilde \bfy}_1 \le m(2m\delta (\Dy) + 1)^m \\
        \forall\ \tilde \bfy\ \text{optimal to \cref{one-b}}\ \exists\ \hat \bfy\ \text{optimal to \cref{one-a} s.t.}\ \norm{\hat \bfy - \tilde \bfy}_1 \le m(2m\delta (\Dy) + 1)^m 
    \end{align}
    \end{subequations}
    where $m$ is the number of rows in matrix $\Dy$, i.e., the number of constraints in the problem.   
\end{proposition}
In \citet{eisenbrand2018proximity}, the result  is presented for programs in the standard form, i.e., with constraints $Ax=b$ and $\bfx\geq0$. The same can be converted to the form used in \cref{one-ab} by adding variables $\bfx^+,\bfx^-$ and $s$ such that $A(\bfx^+ - \bfx^-) + s = b$ and $\bfx^+, \bfx^-, s \geq 0$. Thus, we have $2n + m$ variables in the problem where $n$ is the number of elements in $\bfy$, while the number of constraints $m$ remains the same. 

While \cref{thm:eapGty-Lnew} shows that \cref{alg:apx:lin} provides an approximation for the true solution of \cref{eq:Sklbrg-Lin}, there might be cases where the bound $n \Delta(\Dy)$ is large and therefore not very helpful in proving that the objective $f^\am$ obtained from \cref{alg:apx:lin} is proximal to the true objective $f^\om$. This can happen when $n_y$, the number of variables in the follower's problem of  \cref{eq:Sklbrg-Lin}, is large. 
\cref{thm:E-W} can be useful in such cases because the proximity bound defined here depends on the number of constraints $m$ rather than the number of variables. 
The following theorem utilises \cref{thm:E-W} to prove that the solution $(\bfx^\am,\bfy^\am)$ obtained from \cref{alg:apx:lin} approximates the true solution $(\bfx^\om,\bfy^\om)$.

\begin{theorem} \label{thm:eapGty-L1new}
  Assume $\gy(\bfx)$ is integer valued for every feasible $\bfx$.
  Now, if the optimal objective obtained by the \emph{leader} in \cref{eq:Sklbrg-Lin} be $f^\om$, 
  the following bound holds for the objective $f^\am$ obtained using \cref{alg:apx:lin}: 
 \begin{subequations}
 \begin{align}
   f^\am \quad&\le\quad f^\om + 2\norm{\dq}_\infty m(2m\delta (\Dy) + 1)^m  \label{thm:eaGty-L1} 
\end{align}
\end{subequations}
  where $m$ is the number of rows in matrix $\Dy$. 
\end{theorem}

\begin{proof}[Proof of \cref{thm:eapGty-L1new}]
  Let $\hat \bfy^\om$ be a continuous minimizer of the lower-level problem, given the upper-level decision $\bfx^\om$. 
  We know that from \cref{thm:E-W} that there exists $\norm{{\bfy}^\om - \hat \bfy^\om}_1 \le m(2m\delta (\Dy)+1)^m$, where $ \bfy^\om $ is feasible and optimal to the unrelaxed follower's problem.
  We note that there could be multiple optimal solutions to the follower, not all of them need to be proximal to $\hat \bfy^\om$.
  However, the leader is indifferent among all such follower's optimal solutions because the only term in which the follower's variables appear on the leader's objective functions are the $\dq^\top \bfy$ term, which evaluates to the same value for all follower-optimal solutions.
  Thus, without loss of generality, we can work only with \emph{that} follower-optimal solution that is closest to $\hat \bfy^\om$.
  Now, we have, 
  \begin{subequations}
  \begin{align}
       \hx (\bfx^\am) + \dq^\top \hat \bfy ^\am
       \quad&\leq\quad
       \hx (\bfx^\om) + \dq^\top \hat \bfy^\om \\
       \quad&=\quad 
       \hx (\bfx^\om) + \dq^\top  {\bfy}^\om + 
        {\dq}^\top({\hat \bfy^\om - \bfy^\om}) \\
       \quad&\leq\quad 
       \hx (\bfx^\om) + \dq^\top  {\bfy}^\om + \norm{\dq}_\infty \norm{\hat \bfy^\om - \bfy^\om}_1 \\
            \quad&\leq\quad
      \hx (\bfx^\om) + \dq^\top \bfy^\om + \norm{\dq}_\infty m(2m\delta (\Dy)+1)^m \\
      \quad&=\quad
      f^\om + \norm{\dq}_\infty m(2m\delta (\Dy)+1)^m
    \end{align} \label{eq:t2-1m}
  \end{subequations}
where $f^\om = \hx (\bfx^\om) + \px (\bfy ^\om) $ is the optimum objective value. 
The first inequality in \cref{eq:t2-1m} is due to the fact that $(\bfx^\am, \hat \bfy^\am)$ is feasible and optimal to the FRV, while $(\bfx^\om, \hat \bfy^\om)$ is a feasible point to the FRV.
The equality in the next line is obtained by adding and subtracting $\dq^\top \bfy^\om$.
The inequality in the next line is due to the application of the H\"older's inequality, that $\bfu^\top \bfv \le \norm{\bfu}_\infty \norm {\bfv}_1$. 
The following inequality is due to \cref{thm:E-W}. 
The last equality is because $\bfy^\om$ is the follower's optimal solution.

However, $(\bfx^\am, \hat \bfy^\am)$ is not feasible to \cref{eq:Sklbrg-Lin}. 
The corresponding feasible point is $(\bfx^\am, \bfy^\am)$. 
Given $\bfx^\am$, $\hat \bfy^\am$ is the continuous optimum to the follower's problem, we know from \cref{thm:E-W} that there exists $\tilde{\bfy}^\am$, an integer optimum to the follower's problem such that $\norm{\hat \bfy^\am - \tilde{\bfy}^\am}_1 \le m(2m\delta (\Dy)+1)^m$. 
We note that the integer optimum obtained in \cref{alg:apx:lin:foll} of \cref{alg:apx:lin}, notated as $\bfy^\am$, is not guaranteed to satisfy the bound provided by \cref{thm:E-W}.
However, we know that $\dq^\top \bfy^\am = \dq^\top \tilde{\bfy}^\am$.
Thus,
\begin{subequations}
  \begin{align}
       \hx (\bfx^\am) + \dq^\top \bfy ^\am
       \quad&=\quad
       \hx (\bfx^\am) + \dq^\top \hat \bfy^\am + \dq^\top \bfy ^\am - \dq^\top \hat \bfy^\am 
       \\
            \quad&\le\quad
      f^\om + \norm{\dq}_\infty m(2m\delta (\Dy)+1)^m + \dq^\top \bfy ^\am - \dq^\top\tilde \bfy^\am + \dq^\top\tilde \bfy^\am - \dq^\top \hat \bfy^\am 
            \\
       \quad&\leq\quad
       f^\om + \norm{\dq}_\infty m(2m\delta (\Dy)+1)^m + \dq^\top \bfy ^\am - \dq^\top\tilde \bfy^\am + \norm{\dq}_\infty \norm{\tilde \bfy^\am  - \hat \bfy^\am}_1
            \\
       \quad&\leq\quad
       f^\om + \norm{\dq}_\infty m(2m\delta (\Dy)+1)^m  + \norm{\dq}_\infty m(2m\delta (\Dy)+1)^m
            \\
       \quad&=\quad
       f^\om + 2\norm{\dq}_\infty m(2m\delta (\Dy)+1)^m 
  \end{align} \label{eq:t2-2m} 
\end{subequations}
The equality in the first line is because we add and subtract $\dq^\top \hat \bfy^\am$. 
The inequality in the following line is due to \cref{eq:t2-1m}. 
The following inequality is due to the H\"older's inequality.
The following inequality is due to \cref{thm:E-W}.
The final equality is because the last two terms are the same, completing the proof. 
\end{proof}
\begin{remark}
    The results for an integer-linear lower level requires that the leader's and the follower's objectives are misaligned in the follower's variables. In other words, we needed $\dx(\cdot) = - \dy(\cdot)$. 
    However, no such assumption was required when the lower level was an unconstrained integer convex quadratic minimization problem. 
    This is because of the fundamental difference in the type of proximity results we have in both the cases. In the context of convex quadratic programs, the proximity results gave a bound of the maximum distance between the (unique) continuous minimizer and the {\em farthest} integer minimizer. In contrast, for integer linear programs, the proximity results gave a bound of the maximum distance between any continuous minimizer and its {\em closest} integer minimizer.
\end{remark}

\section{Computational Experiments}
To evaluate the performance of \cref{alg:apx} we run the algorithm on a test bed of instances. 
Due to the absence of standard exact solvers that are customized to solve bilevel programs with a pure-integer convex-quadratic lower level, we implement a complete brute force-based enumeration algorithm that will provide the exact solution. 
We compare the solution time of such an exact algorithm with our approximate algorithm, and also compare the quality of solutions, \emph{i.e.}, the leader's objective function value.
\emph{A priori}, it is obvious that the exact algorithm will take more time and provide better quality solutions (lower objective function value), compared to our approximate algorithm.
The primary insight that the computational tests are expected to provide is the loss in solution quality due to the gain in solution speed in practice, over and above the promised theoretical guarantees in \cref{thm:eapGty,cor:eapGty,thm:apx}.

\subsection{Generation of the testbed. }
For our computational tests, we generate a testbed of 600 instances. 
Of this, there are 300 unique version of the leader's and the follower's problem parameters. 
For each of the 300 unique parameter values, we have one instance, where the follower is assumed to act in an optimistic manner, and another instance where the follower is assumed to act in a pessimistic manner.
All the instances have the following structure. We use the notation in \cref{eq:Sklbrg-Quad-O} to describe the instances.
\begin{enumerate}
  \item In all instances, $\hx$ and $\dx$ are chosen to be linear functions of $\bfx$ and $\bfy$ respectively. The entries of the linear coefficients are randomly generated to be between $-5 $ and $5$ (both inclusive), with equal probability. 
  \item The number of leader variables is always $10$ in all the instances. \emph{i.e., }$n_x = 10$.
  \item The constraints $\gx(\bfx) \le 0$ are generated to be a polyhedron.  The precise form of the constraints is 
    \begin{align*}
      A \bfx \quad&\leq\quad b \\
      \bfx \quad&\in\quad \{0, 1\}^{n_x}. 
    \end{align*}
    In other words, the leader's decision variables are constrained to be a subset of the vertices of the unit cube. 
    The matrix $A$ has each of its entries either $-1$, $0$ or $1$. 
    Moreover, to ensure that the leader is feasible, we randomly generate a binary vector, $\bar \bfx$, and choose $b = A\bar \bfx$, so that at least this one point $\bfx = \bar\bfx$ is feasible for the leader. 
  \item In all instances, the follower parameters, $\Cy$ and $\dy$ are chosen to be matrices of appropriate sizes whose entries  are randomly generated to be between $-9 $ and $9$ (both inclusive), with equal probability. 
  \item We use three different methods to generate the positive-definite matrix $\Qy$ appearing in the objective of the follower. 
    \begin{itemize}
      \item{\bf Diagonal. } In these set of instances, we choose $\Qy$ as a diagonal matrix with positive entries in the diagonal. Each of the diagonal elements entries is randomly generated to be between $1$ and $9$, all equally likely. All off-diagonal entries are $0$. Clearly, by construction, these matrices are positive definite.
      \item{\bf Cholesky-based. } In these set of instances, we randomly generate a matrix $R$ whose entries are randomly generated to be $-1$, $0$ or $1$ with equal probability. 
        We then construct $\Qy = R^\top R + I$, where $I$ is the identity matrix. 
        One can readily observe that $R^\top R$ will be a positive semi-definite matrix. 
        We add  an identity matrix to this to ensure that all eigenvalues are bounded away from $0$, guaranteeing positive definiteness.
        Since, $R$ has all its entries as integers, $\Qy$ also has all its entries as integers.
      \item{\bf Bounded Eigenvalues. }  While the diagonally generated $\Qy$  had all its eigenvalues between $1$ and $9$, there was no guarantee on the eigen values of $\Qy$, when generated using the Cholesky-based method.
        This lead to some of the matrices having very large eigen values. 
        In this method, we generate a diagonal matrix $D$ with all its diagonal entries being randomly chosen between $1$ and $9$. 
        Then, using an implementation of the Haar's algorithm \citep{Mezzadri2006}, to generates a random orthogonal matrix, $U$. 
        Now, we define $\Qy = UDU^\top$.
        Clearly, the eigenvalues of $\Qy$ are the diagonal elements of $D$, which are all between $1$ and $9$.
        Now, the eigenvales of $\Qy$ are bounded, but the random orthogonal matrix $U$ need not (and in general does not) have integer entries. 
        Thus, the matrix $\Qy$ generated using this method does not have integer entries. 
        We store each of the entries up to fifteen decimal places. 
        Bounds due to Ger\v{s}gorin's theorem \citep[Chapter 6]{Horn2012} imply that $\Qy$ does remain positive definite, with eigen values sufficiently close to the diagonal entries of $D$,  despite the truncation and floating point errors. 
    \end{itemize}
  \item For each of the three methods to generate the positive-definite matrix $\Qy$, we choose two sizes, $n_y = 10$ and $n_y = 20$. For each of the six combinations of $\Qy$ generation method and $n_y$, we generate 50 instances. This results in a total of $50\times6\times 2=600$ instances, including the optimistic and pessimistic versions.
 \item We also note that the run time for the {\em Diagonal} instances, {\em i.e.,} instance where $\Qy$ is a diagonal matrices are particularly fast to solve. This is because, the follower's optimization problem is close to trivial in this case. One has to solve $n_y$ number of single dimensional integer convex quadratic minimization problem, and the solution to that is the rounding of the continuous solution. These instances are of particular interest to systematically evaluate the sub-optimality of the approximation algorithm, rather than evaluate the speed. 
\end{enumerate}

\subsection{The testing procedure. }
For each of the 600 instances, we compute the exact solution using a brute-force enumeration algorithm. In this enumeration, we generate all $2^{n_x} = 2^{10} = 1024$ binary vectors, and check if they satisfy all the constraints $Ax \le b$.
If a vector $\bar \bfx$ satisfies all these constraints, then we solve the follower's optimization problem, given the leader's decision $\bfx$, keeping in mind the optimistic/pessimistic assumptions.
The latter is achieved in the solver of our choice, Gurobi \citep{GurobiOptimization2023}, using its ability to handle a \emph{secondary objective function}. 
The secondary objective function (the leader's objective function) is minimized/maximized , among the possibly multiple optimal solutions to the primary objective function, which is the follower's objective function.
This is iterated over every feasible point of the leader's problem. 
A maximum time limit of $120$ seconds is provided. 
If the problem is completed within the time limit, the total time taken is stored. 
The best objective obtained is always stored. 

Then, we run our approximate algorithm. The FRV is a mixed-integer linear program. We store the time taken to solve the FRV. Then, given the leader's solution to the FRV, we solve the follower's problem. This is a pure-integer convex quadratic program. 
Moreover, the follower's best response problem is solved, keeping in mind the optimistic/pessimistic assumptions, and using Gurobi's secondary objective function feature.
We store the time taken to solve the follower's response problem too. 
To evaluate the performance as per \cref{thm:eapGty,cor:eapGty}, the time taken is the sum of the time taken to solve the FRV and to solve the follower's best response. 
To evaluate the performance as per \cref{thm:apx}, the time taken is only the time taken to solve the FRV. 
We also store the objective value of the leader as obtained from the approximate solution.

\subsection{Results. }

The results of our computational experiments are presented in this section. We analyze the performance of our approximation algorithm in terms of two key metrics: time efficiency and solution quality.

\subsubsection{Computational Time Analysis}
To evaluate the time efficiency of \cref{alg:apx}, we compare the times taken by the brute-force enumeration algorithm described above, with \cref{alg:apx}. This comparison is visualized in \cref{fig:performance_profile}. Since the instances with a diagonal $\Qy$ are significantly faster to solve due to a near-trivial follower's problem (see point 7 in Section 6.1) in both the exact algorithm and \cref{alg:apx}, they have been excluded from this analysis. Therefore, the computational time analysis is done only based on the remaining 400 instances. 

\begin{figure}
  \centering
  \includegraphics[width=0.8\textwidth]{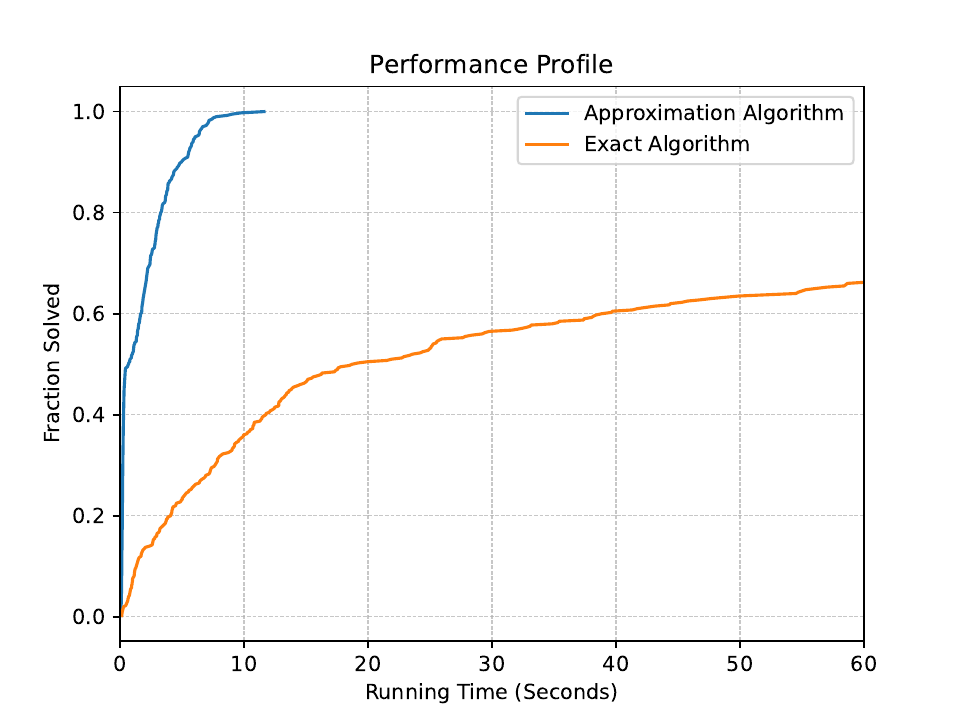}
  \caption{Performance Profile (Approximation vs. Exact Algorithm)}
  \label{fig:performance_profile}   
\end{figure}

Across all the test instances, \cref{alg:apx} consistently required significantly less computational time compared to the exact algorithm. As evident in \cref{fig:performance_profile}, \cref{alg:apx} solved all 400 instances in under 12 seconds. The exact algorithm solved only $66\%$ of all the instances in 60 seconds. We also observe that the exact algorithm has a very long tail, with $80\%$ of the instances solved in 174 seconds, $85\%$ solved in 263 seconds, $90\%$ in 411 seconds, $95\%$ in 922 seconds, and 100$\%$ in 1674 seconds (almost 28 minutes).

In contrast, we see that in $88.75\%$ of the instances, \cref{alg:apx} completed in less than 0.1 times the time taken by the exact algorithm, showing a substantial speedup. For larger problem instances where the run time of the exact algorithm exceeded 100 seconds, \cref{alg:apx} showed dramatic improvements, often completing the computation in less than 2 seconds. Furthermore, \cref{alg:apx} shows consistent efficiency, with an average run time of 1.78 seconds per instance and maintaining computational times close to or below 1 second for 57.75$\%$ of the instances. This predictability and speed make our approximation algorithm highly suitable for practical applications where time efficiency is crucial. 

\subsubsection{Solution Quality Analysis}
To assess the quality of the solutions obtained by \cref{alg:apx}, we measured the difference in the optimal objective function value between the approximate and exact algorithms, $\Delta\fx = \fx^\am - \fx^\om$.
Note that the coefficient of the leader's objective function, $\hx$ and $\dx$ are all integers. Since $\bfx$ and $\bfy$ are constrained to be integers too, the objective value of any feasible solution is an integer. 
Thus, the difference between the approximate and the exact solution is also necessarily an integer. 
The distribution of this difference $\Delta\fx$ over the 600 test instances is depicted in \cref{fig:quality_histogram}. 

\begin{figure}
  \centering
  \includegraphics[width=0.8\textwidth]{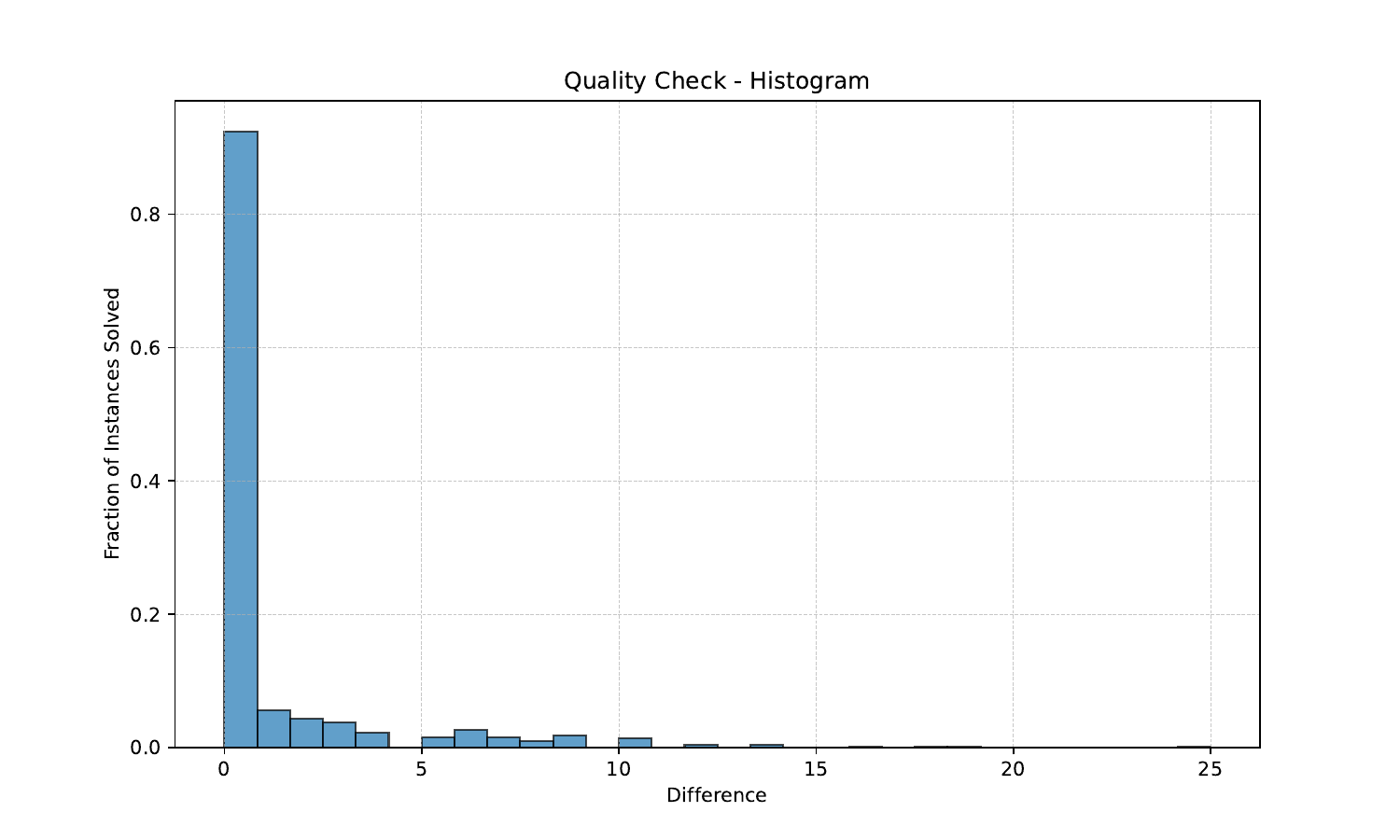}
  \caption{Histogram of $\Delta\fx$}
  \label{fig:quality_histogram}   
\end{figure}

\cref{fig:quality_histogram} reveals that the majority of the solutions from \cref{alg:apx} are close to the optimal solutions obtained by the brute-force exact algorithm. 
In $77\%$ of the instances, $\Delta\fx = 0$.
This means that in over three-fourth of all instances, \cref{alg:apx} solves the problem correctly.
$91\%$ of the instances have $\Delta\fx \le 5$, indicating that the deviation from the exact optimal objective function value is minimal. This further showcases the approximation algorithm's reliability in maintaining solution quality. Instances with higher $\Delta\fx$ values are relatively rare, with only $2.5\%$ of the instances having $\Delta\fx \ge 10$, emphasizing that the algorithm performs well in most cases. 
We see that \cref{alg:apx} is robust across a wide range of problem instances, indicating its suitability for diverse scenarios in practical settings. 

\bibliographystyle{plainnat}
\bibliography{refSmall,litRev}
\end{document}